\documentclass[preprint,3p]{elsarticle}
\usepackage[english]{babel}
\usepackage{amssymb}
\usepackage{amsmath}
\usepackage{amsthm}
\usepackage{bm}
\usepackage{float}
\usepackage{graphicx}
\usepackage[caption]{subfig}
\usepackage{caption}
\usepackage{units}
\usepackage{lineno}
\usepackage{multirow}
\usepackage[strings]{underscore}

\newtheorem{theorem}{Theorem}[section]

\newtheorem{lemma}{Lemma}[section]

\setcitestyle{square, numbers}

\newcommand{\mtrx}[1]{\mathbf{#1}}	
\newcommand{\gmtrx}[1]{\bm{#1}}	
\newcommand{\me}{\mathrm{e}}
\newcommand{\mi}{\mathrm{i}}
\begin{document}
\begin{frontmatter}

\title{Explicit Third-Order Model Reduction Formulas \\  for General Nonlinear Mechanical Systems}

\author[mymainaddress]{Zsolt Veraszto}
\ead{zsoltv@student.ethz.ch}
\author[mymainaddress]{Sten Ponsioen\corref{mycorrespondingauthor}}
\ead{stenp@ethz.ch}
\author[mymainaddress]{George Haller}
\ead{georgehaller@ethz.ch}

\address[mymainaddress]{Institute for Mechanical Systems \\ ETH Z{\"u}rich, Leonhardstrasse 21, 8092 Z{\"u}rich, Switzerland}
\cortext[mycorrespondingauthor]{Corresponding author.}

\begin{abstract}
{\it
For general nonlinear mechanical systems, we derive closed-form, reduced-order models up to cubic order based on rigorous invariant manifold results. For conservative systems, the reduction is based on Lyapunov Subcenter Manifold (LSM) theory, whereas for damped-forced systems, we use  Spectral Submanifold (SSM) theory. To evaluate our explicit formulas for the reduced model, no coordinate changes are required beyond an initial linear one. The reduced-order models we derive are simple and depend only on physical and modal parameters, allowing us to extract fundamental characteristics, such as backbone curves and forced-response curves, of multi-degree-of-freedom mechanical systems. To numerically verify the accuracy of the reduced models, we test the reduction formulas on several mechanical systems, including a higher-dimensional nonlinear Timoshenko beam.
}
\end{abstract}


\end{frontmatter}
\section{Introduction}
The dimensional reduction of nonlinear systems in structural dynamics has two underlying drivers. One is the necessity to decrease computational time, as computing a reduced-order model can be orders of magnitudes faster than solving the full system. The second driver is model identification, which is easier to carry out when a reduced number of parameters needs to be fitted to experimental or simulation data.

The appealing properties of modal decomposition and reduction in linear systems motivate the exploration of a similar decomposition and reduction for nonlinear systems. The latter can be accomplished in a mathematically exact fashion if the spectral subspaces of the linearized system can be continued as invariant manifolds of equal dimension, tangent to those spectral subspaces at a fixed point. Moreover, if these manifolds are unique and analytic, they can be sought in the form of Taylor expansions near the fixed point. As a consequence, the dynamics on these manifolds will serve as a mathematically exact reduced-order model that can be constructed up to any required order of accuracy.

In this work, we provide ready-to-use explicit formulas for third-order reduced models of autonomous and non-autonomous mechanical systems with general second- and third-order polynomial nonlinearities, without transforming the second-order equations of motion to a first-order form. For conservative systems, this third-order model reduction is performed onto Lyapunov Subcenter Manifolds (LSMs) \cite{lsm}, whereas for damped-forced systems, we perform the reduction onto Spectral Submanifolds (SSMs) \cite{nnmssm}, which are mathematically rigorous versions of the nonlinear normal modes proposed first by Shaw and Pierre \cite{shaw1993normal}. To obtain higher-order LSM- and SSM-reduced models, one may use the automated numerical reduction procedure developed in \cite{Ponsioen2018}.

Alternative approaches to nonlinear model reduction in mechanical systems include the method of normal forms (\cite{touze2006nonlinear,neild2010applying,neild2015use} ) and the method of modal derivatives \cite{modalderivatives}. Normal forms involve a recursive simplification of the full system via a sequence of nonlinear coordinate changes near a fixed point. Truncated at any finite order, the normal form admits invariant subspaces to which one may reduce the dynamics. It is, however, generally unclear from this approach what will happen to these invariant subspaces under the addition of the remaining, unnormalized terms. Furthermore, the reduced model is available in coordinates that are nonlinear functions of the original coordinates and hence their physical meaning is not immediate. Finally, the procedure assumes the knowledge of a full linear modal decomposition at the start, which is not necessarily available for large systems.

The method of normal forms has nevertheless been applied to equations of motion in second-order form including periodic forcing (cf. Neild and Wagg \cite{neild2010applying} and Neild et al. \cite{neild2015use}). As an advantage, the method allows for the extraction of the steady-state response in addition to backbone curves in the unforced and undamped limit. As opposed to SSM theory, however, the method does not provide a mathematically exact reduced-order model and requires additional smallness assumptions on the nonlinear and damping terms.

In earlier work, Touz{\'e} and Amabili \cite{touze2006nonlinear} propose a normal form procedure for the construction of reduced-order models for geometrically nonlinear vibrations of thin structures. They introduce the external forcing directly into the normal form along non-physical, curvilinear coordinates. In this work, we prove that at leading-order, the SSM-reduced system of a periodically forced mechanical system can be seen as the autonomous SSM-reduced system with the modal-participation factor (see G{\'e}radin and Rixen \cite{Geradin2014}) of the first mode added to the reduced system, therefore justifying the Touz{\'e}--Amabili approach. Our proof utilizes recent work by Breunung and Haller \cite{thomas} on time-periodic SSMs in mechanical systems.

A more recent application of the method of normal forms for model-order reduction purposes can be found in the work of Denis et al. \cite{equivduffing}. They observe that, in the absence of internal resonances, a reduced-order model for each mode can be approximated by a Duffing oscillator. The method requires full knowledge of the linear modal decomposition and excludes resonances up to any order in order to make the unnormalized tail of the asymptotic expansion arbitrarily small. In this work, we confirm this idea mathematically using LSM theory and show that for a conservative system with symmetric cubic nonlinearities, a third-order reduced model can indeed be sought as a Duffing oscillator, without using a modal transformation.

The other alternative to nonlinear model reduction is the method of modal derivatives \cite{modalderivatives}, which assumes the existence of quadratic invariant manifolds in the configuration space (space of positions). While such a manifold can formally be sought, it generally does not actually exist in the configuration space but in the phase space. Accordingly, modal derivatives only give an accurate reduced-order model in slow-fast systems in which spectral submanifolds have a weak dependence on the velocities (see  \cite{slowfast}).

In our present development of explicit, cubic reduction formulas, we circumvent the issues with normal forms and modal derivatives and provide generally applicable, explicit formulas for nonlinear reduced-order models. We also illustrate the use of our reduction formulas on different mechanical systems varying from simple to complex.
%
%

\section{Conservative systems}

\subsection{Setup}

We consider $(n+1)$-degree-of-freedom, undamped nonlinear autonomous mechanical systems of the form
  \begin{align}
    \ddot x + \omega^2 x+P(x,\mtrx y)&=0,
    \label{systemx}\\
    \mtrx{M\ddot y + K y}+\mtrx{Q}(x,\mtrx{y})&=\mtrx{0},
    \label{systemy}
  \end{align}
where $x$ is a scalar, $\mtrx y$ is an $n$-vector; $P$ and $\mtrx{Q}$ are nonlinear, polynomial functions of $x$ and $\mtrx y$ up to order $2 \leq l\leq 3$; $\mtrx M$ and $\mtrx K$ are the mass and stiffness matrices corresponding to the $n$-degree-of-freedom generalized coordinate vector $\mtrx y$, while $m$ and $k$ are the modal mass and stiffness coefficients corresponding to $x$. We refer to the $x$ variable as the modeling variable and $\mtrx y$ as the vector of non-modeling variables. Obtaining the form (\ref{systemx})-(\ref{systemy}) for a mechanical system requires the use of a linear transformation that decouples the modeling variable from the rest. To this end, only the mode shape associated with the $x$-mode has to be explicitly known; the remaining $n$ columns of the matrix of the linear transformation can be chosen as an arbitrary basis from the $n$-dimensional plane orthogonal to the known mode shape in $\mathbb{R}^{n+1}$.

We will use the notation
\begin{equation}
\omega=\sqrt{\frac{k}{m}},\quad \gmtrx{\Omega}^2_\text{p} = {\mtrx{M}^{-1}\mtrx{K}},
\end{equation}
where $\omega$ is the natural frequency belonging to the modeling mode $x$ with the imaginary pair of eigenvalues $\pm\mi\omega$. The non-modeling modes belong to the purely imaginary eigenvalues $\pm\mi\omega_1,\ldots,\pm\mi\omega_n$, with $\omega_k$ denoting the $k^\text{th}$ natural frequency. Our notation for the coefficients of the $l^\text{th}$-order polynomials $P$ and $\mtrx{Q}$ is
\begin{gather}
P(x,\mtrx{y})=\sum_{j+|\mtrx{k}|=2}^{l}p_{j\mtrx{k}}x^j y_1^{k_1}\ldots y_n^{k_n},\quad \mtrx{Q}(x,\mtrx{y})=\sum_{j+|\mtrx{k}|=2}^{l}\mtrx{q}_{j\mtrx{k}}x^j y_1^{k_1}\ldots y_n^{k_n}, \label{notation_1}\\
p_{j\mtrx{k}}\in\mathbb{R},\quad \mtrx{q}_{j\mtrx{k}}\in\mathbb{R}^n,\quad \mtrx{k}\in\mathbb{N}^n. \nonumber
\end{gather}
For the integer vectors $\mtrx{k}\in\mathbb{N}^n$, specific values we will often need are
\begin{equation}
\mtrx{0}=(0,\ldots,0)\in\mathbb{N}^n,\quad \mtrx{e}_j=(0,\ldots,0,\underset{j^\text{th}}{1},0,\ldots,0)\in\mathbb{N}^n, \nonumber
\end{equation}
denoting the identically zero index vector and the $j^\text{th}$ unit vector, respectively. Additionally, we introduce the notation
\begin{equation}
\mtrx{p}_{j\mtrx{I}}=\left[p_{j\mtrx{e}_1},\ldots,p_{j\mtrx{e}_n}\right]^\top\in\mathbb{R}^{n},\quad
{\mtrx{Q}}_{j\mtrx{I}}=\left[\mtrx{q}_{j\mtrx{e}_1},\ldots,\mtrx{q}_{j\mtrx{e}_n}\right]\in\mathbb{R}^{n\times n}.
\end{equation}

\subsection{The LSM-reduced model in the general case \label{sec:lsm_general}}
As shown by Kelley \cite{lsm}, under the non-resonance conditions
\begin{equation}
  \frac{\omega_i}{\omega}\notin \mathbb{Z},\quad i=1,\dots,n,
  \label{lsmnonresonance}
\end{equation}
a unique, analytic, two-dimensional invariant manifold
\begin{equation}
  \mathcal{W}=\lbrace (x,\dot x,\mtrx y, \dot{\mtrx{y}})\in\mathbb{R}^{2n+2},\,\mtrx y=\mtrx w (x,\dot x), \dot{\mtrx{y}} = \dot{\mtrx{w}}(x,\dot x)\rbrace
\end{equation}
exists for system (\ref{systemx})-(\ref{systemy}), tangent to the two-dimensional subspace $\mtrx{y}=\dot{\mtrx{y}}=\mtrx{0}$ at the trivial fixed point $x=\dot{x}=0$, $\mtrx{y}=\dot{\mtrx{y}}=\mtrx{0}$. Often called the Lyapunov subcenter manifold (LSM), $\mathcal{W}$ is known to be filled with periodic orbits of (\ref{systemx})-(\ref{systemy}). In the nonlinear vibrations literature, the amplitude-frequency plot of these periodic orbits is called the (conservative) \emph{backbone curve} of system (\ref{systemx})-(\ref{systemy}) (see, e.g., Londono et al. \cite{londono2015identification}).

Reducing the full dynamics to the invariant manifold $\mathcal{W}$ gives an exact two-dimensional model for the nonlinear dynamics associated with the $x$-mode. Such a model has apparently not been computed for the general class of systems (\ref{systemx})-(\ref{systemy}) in the literature before. We will carry out this computation explicitly below.

As a first step, we derive a universally valid third-order approximation to $\mathcal{W}$ in the form
\begin{equation}
  \mtrx{y} = \mtrx{w}(x,\dot{x})=\sum_{i+j\in \{2,3\}} \mtrx{w}_{ij} x^i \dot x^j+\mathcal{O} (|(x,\dot x)|^4).
  \label{3rdapprox}
\end{equation}
We restrict (\ref{systemx})-(\ref{systemy}) to the LSM by enslaving the $\mtrx{y}$ variable to $(x,\dot{x})$ in equation (\ref{systemx}) as
\begin{equation}
    \ddot x + \omega^2 x+P\left(x,\mtrx w\left(x,\dot x\right)\right)+\mathcal{O} \left(|\left(x,\dot x\right)|^4\right)=0. \label{reducedmodeldefinition}
\end{equation}
More specifically, we obtain the following result.
\begin{theorem}\label{thrm:lsmform}
Under assumption (\ref{lsmnonresonance}), the exact, third-order reduced model for system (\ref{systemx})-(\ref{systemy}) on the LSM, $\mathcal{W}$, takes the form
  \begin{equation}
      \ddot x + \omega^2 x+p_{2\mtrx{0}}x^2+\left(p_{3\mtrx{0}}+ \left<\mtrx{p}_{1\mtrx{I}},\mtrx{w}_{20}\right>\right) x^3 +\left<\mtrx{p}_{1\mtrx{I}},\mtrx{w}_{02}\right> x \dot x^2+\mathcal{O} (|(x,\dot x)|^4)=0,
      \label{lsmreducednonmodal}
  \end{equation}
where
\begin{align}
\mtrx{w}_{20}&=-\left(\gmtrx{\Omega}_\text{p}^2-2\omega^2\mtrx{I}\right)\left(\gmtrx{\Omega}_\text{p}^2-4\omega^2\mtrx{I}\right)^{-1}\gmtrx{\Omega}_\text{p}^{-2}\mtrx{M}^{-1}\mtrx{q}_{2\mtrx{0}}, \nonumber \\
\mtrx{w}_{02}&=2\left(\gmtrx{\Omega}_\text{p}^2-4\omega^2\mtrx{I}\right)^{-1}\gmtrx{\Omega}_\text{p}^{-2}\mtrx{M}^{-1}\mtrx{q}_{2\mtrx{0}}. \nonumber
\end{align}
\end{theorem}
\begin{proof}
This result follows directly upon substitution of the coefficients derived in Lemma \ref{lsmcoeffequationslemma} of Appendix \ref{app:der_lsm}.
\end{proof}


When system (\ref{systemx})-(\ref{systemy}) is lower dimensional and full modal decomposition can be performed in reasonable time, one can also derive explicit formulas for the LSM coefficients, as we show below in Section \ref{modallsmreduction}. These formulas enable a direct comparison of our results with other methods, such as the method of normal forms \cite{touze}, that unavoidably require a full modal transformation of the system.

\subsection{Explicit form of the LSM-reduced model when the non-modeling modes are available}
\label{modallsmreduction}

After a modal transformation of the form $\mtrx{y}=\gmtrx{\Phi \eta}$, system (\ref{systemx})-(\ref{systemy}) can be written as
\begin{align}
  \ddot x+\omega^2 x+ R(x,\gmtrx\eta) &=0,
  \label{modalx}\\
  \gmtrx{\ddot\eta}+\gmtrx{\Omega}^2 \gmtrx{\eta}+\mtrx{S}(x,\gmtrx{\eta})&=\mtrx 0,
  \label{modaly}
\end{align}
where $\omega$ is the natural frequency of the modeling variable $x$, and the diagonal matrix $\gmtrx \Omega$ contains the natural frequencies of the non-modeling variables $\mtrx y$. The $l^\text{th}$-order polynomials $R$ and $\mtrx{S}$ are given by
\begin{gather}
R(x,\gmtrx{\eta})=\sum_{j+|\mtrx{k}|=2}^{l}r_{j\mtrx{k}}x^j \eta_1^{k_1}\ldots \eta_n^{k_n},\quad \mtrx{S}(x,\gmtrx{\eta})=\sum_{j+|\mtrx{k}|=2}^{l}\mtrx{s}_{j\mtrx{k}}x^j \eta_1^{k_1}\ldots \eta_n^{k_n},\label{notation_2}\\
r_{j\mtrx{k}}\in\mathbb{R},\quad \mtrx{s}_{j\mtrx{k}}\in\mathbb{R}^n,\quad \mtrx{k}\in\mathbb{N}^n. \nonumber
\end{gather}
Similarly to Section \ref{sec:lsm_general}, we introduce the notation
\begin{equation}
\mtrx{r}_{j\mtrx{I}}:=\left[r_{j\mtrx{e}_1},\ldots,r_{j\mtrx{e}_n}\right]^\top\in\mathbb{R}^{n},\quad
{\mtrx{S}}_{j\mtrx{I}}:=\left[\mtrx{s}_{j\mtrx{e}_1},\ldots,\mtrx{s}_{j\mtrx{e}_n}\right]\in\mathbb{R}^{n\times n}.
\end{equation}
We now derive the third-order approximation to $\mathcal{W}$ in modal coordinates in the form
\begin{equation}
 \gmtrx{\eta}=\tilde{\mtrx{w}}(x,\dot{x})=\sum_{i+j\in \{2,3\}} \tilde{\mtrx{w}}_{ij} x^i \dot x^j+\mathcal{O} (|(x,\dot x)|^4).\label{3rdapproxmodal}
\end{equation}
We restrict (\ref{modalx})-(\ref{modaly}) to the LSM by enslaving the $\gmtrx{\eta}$ variable to $(x,\dot{x})$ in eq. (\ref{modalx}), which gives
   \begin{equation}
       \ddot x + \omega^2 x+R(x,  \tilde{\mtrx{w}}(x,\dot x))+\mathcal{O} (|(x,\dot x)|^4)=0.
   \end{equation}
We then obtain the following result.
\begin{theorem} \label{thrmROM}
Under assumption (\ref{lsmnonresonance}), the exact, third-order reduced model of (\ref{modalx})-(\ref{modaly}) on the LSM takes the form
\begin{equation}
     \begin{split}
       &\ddot x + \omega^2 x+r_{2\mtrx{0}}x^2+\left(r_{3\mtrx{0}}-\left<\mtrx r_{1\mtrx{I}}, (\gmtrx{\Omega}^2\mtrx{D}_4)^{-1}\mtrx{D}_2 \mtrx s_{2\mtrx{0}}\right>\right) x^3
       +2 \left<\mtrx r_{1\mtrx{I}}, (\gmtrx{\Omega}^2\mtrx{D}_4)^{-1} \mtrx s_{2\mtrx{0}}\right> x \dot x^2+\mathcal{O} (|(x,\dot x)|^4)=0,
     \end{split} \label{lsmreducedimmediatelyuseable}
\end{equation}
where
\begin{equation}
 \mtrx{D}_2 = (\gmtrx \Omega ^2-2\omega^2 \mtrx I), \quad   \mtrx{D}_4= (\gmtrx \Omega ^2-4\omega^2 \mtrx I). \nonumber
\end{equation}

%
\end{theorem}

\begin{proof}
This result is directly obtained upon substitution of the coefficients derived in Lemma \ref{lsmcoefflemma} of Appendix \ref{app:lsm_avai}.
\end{proof}

Although similar results can be obtained using normal form theory, see e.g. Denis et al.\ \cite{equivduffing}, a clear advantage of the LSM reduction method is that it requires no nonlinear change of coordinates. As a consequence, the reduced model (\ref{lsmreducedimmediatelyuseable}) keeps the original modeling coordinate $x$ as a physically meaningful modal coordinate. Additionally, the coefficients of the reduced model (\ref{lsmreducedimmediatelyuseable}) are simple to compute numerically.

Note that equation (\ref{lsmreducedimmediatelyuseable}) takes the form of a Duffing-oscillator if
  \begin{equation}
 r_{2\mtrx{0}}=0,\quad \left<\mtrx r_{1\mtrx{I}}, (\gmtrx{\Omega}^2\mtrx{D}_4)^{-1} \mtrx s_{2\mtrx{0}}\right>=0.
  \end{equation}
This condition is satisfied, for instance, for an even potential function that leads to vanishing  second-order nonlinearities. In that case, the reduced model, accurate up to $\mathcal{O} (x^3)$, is given by
    \begin{equation}
      \ddot x + \omega^2 x+r_{3\mtrx{0}} x^3+\mathcal{O} (x^4)=0.
    \end{equation}
Consequently,  for a conservative system with symmetric cubic nonlinearities, a third-order reduced model can immediately be sought as a Duffing-oscillator. For such systems, therefore, the numerical and experimental procedure of Olivier et al. \cite{equivduffing} is justified when they define an equivalent Duffing-oscillator for reduced models obtained from a normal form analysis.

  \subsection{Reduced and truncated dynamics on the LSM}
  The full reduced dynamics on the LSM, a symplectic invariant manifold of a canonical Hamiltonian system, is necessarily Hamiltonian. However, system (\ref{lsmreducednonmodal}) is a truncation of the full reduced dynamics. Therefore, it is a priori unclear if (\ref{lsmreducednonmodal}) is Hamiltonian or not.

  In the following, we show that the truncated reduced model on the LSM is a conservative nonlinear oscillator with closed orbits, and therefore accurately captures the long-term qualitative dynamics on the LSM. As a first step, we rewrite the third-order reduced model (\ref{lsmreducednonmodal}) as a two-dimensional first-order system
  \begin{equation}
  \begin{split}
  \dot{x}=&y, \\
  \dot{y}=&-\omega^2 x -\alpha x^2 -\beta x^3 - \gamma xy^2. \label{firstorder}
  \end{split}
  \end{equation}
Note that system (\ref{firstorder}) is invariant under the transformation $x\rightarrow x$, $y\rightarrow -y$, $t\rightarrow-t$, and therefore is reversible with respect to the line $y=0$ \cite{Stepanov1960}.
  \begin{theorem}\label{po_existence}
  Let (\ref{firstorder}) be nondegenerate, i.e., assume that the determinant of its Jacobian at $x=y=0$ is positive, and assume that (\ref{firstorder}) is reversible. Then, the $x=y=0$ equilibrium of (\ref{firstorder}) is a center, i.e., system (\ref{firstorder}) has only periodic orbits in a vicinity of the origin.
  \end{theorem}
  \begin{proof}
  See Stepanov and Nemytskii \cite{Stepanov1960} who prove this result for general reversible systems.
  \end{proof}
 Theorem \ref{po_existence} guarantees that locally around the origin, the truncated system qualitatively captures the long-term dynamics on the LSM. For a complete global understanding of the phase space of system (\ref{firstorder}), we now prove that the third-order reduced model is a generalized Hamiltonian system.

  \begin{theorem}\label{thm_ham}
System (\ref{firstorder}) is Hamiltonian with the symplectic form $\omega$ and the Hamiltonian $H$ defined as
  \begin{equation}
 \omega=\me^{-\gamma x^2}dx\wedge dy,\quad H(x,y) = \frac{1}{2}e^{\gamma x^2}y^2+\frac{\omega^2}{2\gamma}e^{\gamma x^2}
  		+ \int e^{\gamma x^2}\left(\alpha x^2 + \beta x^3\right)dx.   \label{Hproof}
  \end{equation}
  \end{theorem}
  \begin{proof}
We derive this result in Appendix \ref{app:ham}.
  \end{proof}
The linearization of system (\ref{firstorder}) reveals that its fixed point, $(x,y)=(0,0)$, is a center in linear approximation. The Hessian of $H(x,y)$, evaluated at $(x,y)=(0,0)$, is
  \begin{equation}
  D^2H(0,0) = \omega^2 > 0,
  \end{equation}
thus we conclude that the truncated reduced-order model is Lyapunov stable. This proves that, beyond being a close pointwise approximation, the truncated model also accurately represents the long-term, qualitative dynamics on the LSM around the fixed point.

\subsection{Comparison with reduced models obtained from the method of modal derivatives}

Although the application of modal derivatives (MDs) is based on an idea similar to the LSM reduction, MDs lack the underlying invariant manifold theory that would ensure robustness or validity of the reduced-order model  \cite{modalderivatives}. The manifold proposed by MDs is sought in the quadratic form $\gmtrx \eta =\gmtrx\Theta x^2$, where $\gmtrx\Theta$ is expressed, according to \cite{slowfast}, as

\begin{equation}
\gmtrx \Theta = -\frac{1}{2}\partial_{\gmtrx \eta} \mtrx S(x,\gmtrx \eta)^{-1}\left. \partial^2_{xx} \mtrx S(x,\gmtrx \eta)\right|_{x=0, \gmtrx \eta=\mtrx 0}, \\
\end{equation}
where
\begin{equation}
\mtrx S(x,\gmtrx \eta) =  \gmtrx{\Omega}^2 \gmtrx{\eta}+\sum_{j+|\mtrx{k}|=2}^{3}\mtrx{s}_{j\mtrx{k}}x^j \eta_1^{k_1}\ldots \eta_n^{k_n}.
\end{equation}
This leads to the coefficient vector

    \begin{equation}
        \gmtrx \Theta=-\gmtrx \Omega^{-2} \mtrx s_{2\mtrx{0}},
    \end{equation}
and ultimately to the reduced model
   \begin{equation}
       \ddot x + \omega_0^2 x+p_{2\mtrx{0}}x^2+\left(p_{3\mtrx{0}}-\left<\mtrx p_{1\mtrx{I}}, \gmtrx \Phi \gmtrx \Omega^{-2} \mtrx s_{2\mtrx{0}}\right>\right) x^3=0.
   \end{equation}

\begin{theorem}\label{mdtheorem}
In the limit of $\Omega_i \rightarrow \infty$, i.e., for large modal gaps between the modeling and non-modeling frequencies, the MD-reduced model approaches the LSM-reduced model. The two methods, however, do not agree in any other nontrivial case. Near resonances among the natural frequencies, the error of the MD-based reduction becomes unbounded.
\end{theorem}
\begin{proof}
We derive this result in Appendix \ref{app:md}.
\end{proof}
By Theorem \ref{mdtheorem}, the error of a model reduction using modal derivatives, instead of the LSM-reduced model, will be small for extremely large spectral gaps. Generally, however, the error introduced by the formal MD approach is substantial and grows unbounded near external resonances.

\subsection{Comparison with reduced models obtained from the method of normal forms }

The method of normal forms, described in detail in \cite{touze}, offers an alternative way to reduce a nonlinear mechanical system to a single mode and infer its hardening or softening behavior. The decoupling of the modeling mode from the rest is achieved through a nonlinear, near-identity change of variables. In that setting, the system is assumed to be in modal coordinates to begin with, therefore throughout this section, we also assume the input system for reduction in the form of equations (\ref{modalx})-(\ref{modaly}).

The key step in the derivation of a normal-form-based reduced model is to introduce a near-identity transformation of the modal coordinate system $\mtrx q=[x;\gmtrx \eta]$ to a curvilinear coordinate system $\mtrx s= [u; \mtrx v]$ of the form

\begin{align}
q_i&=s_i+\sum_{j=1}^{n+1} \sum_{k \geq j}^{n+1}\left(a_{ijk} s_j s_k+b_{ijk} \dot s_j \dot s_k\right)+\sum_{j=1}^{n+1} \sum_{k \geq j}^{n+1} \sum_{l \geq k}^{n+1}  r_{ijkl} s_j s_k s_l +\sum_{j=1}^{n+1} \sum_{k =1}^{n+1} \sum_{l \geq k}^{n+1} u_{ijkl} s_j \dot s_k \dot s_l,
\label{nearidentity1}\\
\dot q_i&=\dot s_i+\sum_{j=1}^{n+1} \sum_{k =1}^{n+1}\gamma_{ijk} s_j \dot s_k+\sum_{j=1}^{n+1} \sum_{k \geq j}^{n+1} \sum_{l \geq k}^{n+1} \mu_{ijkl} s_j s_k s_l +\sum_{j=1}^{n+1} \sum_{k =1}^{n+1} \sum_{l \geq k}^{n+1} \nu_{ijkl} \dot s_j s_k s_l,
\label{nearidentity2}
\end{align}
which constitutes a decoupling up to a given order between the modeling and non-modeling modes. The reduced model obtained by this approach, with the same notation as in (\ref{notation_1}) and (\ref{notation_2}), reads as

\begin{equation}
  \ddot u + \omega^2 u +\left(r_{3\mtrx{0}}-\langle\mtrx r_{1\mtrx{I}}, \left(\gmtrx\Omega^2 \mtrx D_2\right)^{-1}\mtrx s_{2\mtrx{0}} \rangle-\frac{2 r_{2\mtrx{0}}^2}{3 \omega^2} \right) u^3 +2 \left(\langle \mtrx r_{1\mtrx{I}},\gmtrx\Omega^2 \mtrx D_2^{-1}\mtrx s_{20} \rangle-\frac{2 r_{2\mtrx{0}}^2}{3 \omega^4}\right) u \dot u^2=0.
  \label{normalformreducedmodel}
\end{equation}

Although (\ref{normalformreducedmodel}) is similar to the LSM reduced model, it can only be computed if the system is available in modal coordinates. In that case, one has to compute only a small subset of the coefficients in the near-identity transformation (\ref{nearidentity1})-(\ref{nearidentity2}) to obtain (\ref{normalformreducedmodel}). The reduced model correctly identifies hardening or softening behavior up to second order (see Section \ref{backbonesection}).

The method of normal forms, however, does not guarantee that an actual LSM exists in the phase space under addition of higher-order, unnormalized terms in system (\ref{modalx})-(\ref{modaly}). In that sense, a truncated normal-form based reduced model is not justified at the same level of mathematical rigor as the LSM-based reduced model.

\section{Dissipative systems}
\subsection{Setup}
We now consider $(n+1)$-degree-of-freedom, periodically forced, damped nonlinear mechanical systems of the form
    \begin{gather}
        m {\ddot x} + c {\dot x}+ k x + P( x,\dot x, \mtrx y, \mtrx{\dot y})= \varepsilon F_1 \sin(\phi), \label{xeq}\\
        \mtrx M \mtrx{\ddot y} + \mtrx C \mtrx{\dot y}+ \mtrx K \mtrx y +\mtrx Q(x,\dot x, \mtrx y, \mtrx{\dot y})=\varepsilon \mtrx{F}_2 \sin(\phi), \label{yeq}\\
        0<\varepsilon\ll 1,\quad \phi=\Omega t, \nonumber
    \end{gather}
where $x$ is a scalar, $\mtrx y$ is an $n$-vector; $P$ and $\mtrx{Q}$ are now nonlinear, polynomial functions of $x$, $\dot{x}$, $\mtrx y$ and $\dot{\mtrx{y}}$ up to order $2 \leq l\leq 3$; $\mtrx M$, $\mtrx C$ and $\mtrx K$ are the mass, damping and stiffness matrices corresponding to the $n$-degree-of-freedom generalized coordinate vector $\mtrx y$, while $m$, $c$ and $k$ are the modal mass, damping and stiffness coefficients, corresponding to $x$. We assume Rayleigh-damping, which implies that the modal matrices simultaneously diagonalize the mass, the stiffness and the damping matrices. The mono-harmonic forcing vectors $F_1 \sin(\phi)$ and $\mtrx{F}_2 \sin(\phi)$, do not depend on the positions and velocities. In this setting, our notation of the coefficients of the $l^\text{th}$-order polynomials $P$ and $\mtrx{Q}$ is
\begin{gather}
P(x,\dot{x},\mtrx{y},\dot{\mtrx{y}})=\sum_{j+k+|\mtrx{u}|+|\mtrx{v}|=2}^{l}p_{jk\mtrx{u}\mtrx{v}}x^j \dot{x}^k y_1^{u_1}\ldots y_n^{u_n}\dot{y}_1^{v_1}\ldots\dot{y}_n^{v_n}, \\
\mtrx{Q}(x,\dot{x},\mtrx{y},\dot{\mtrx{y}})=\sum_{j+k+|\mtrx{u}|+|\mtrx{v}|=2}^{l}\mtrx{q}_{jk\mtrx{u}\mtrx{v}}x^j \dot{x}^k y_1^{u_1}\ldots y_n^{u_n}\dot{y}_1^{v_1}\ldots\dot{y}_n^{v_n}, \label{notation_ssm}\\
p_{jk\mtrx{u}\mtrx{v}}\in\mathbb{R},\quad \mtrx{q}_{jk\mtrx{u}\mtrx{v}}\in\mathbb{R}^n,\quad \mtrx{u}\in\mathbb{N}^n,\quad \mtrx{v}\in\mathbb{N}^n. \nonumber
\end{gather}
Additionally, we introduce the notation
\begin{equation}
\mtrx{p}_{jk\mtrx{I}\mtrx{0}}:=\left[p_{jk\mtrx{e}_1\mtrx{0}},\ldots,p_{jk\mtrx{e}_n\mtrx{0}}\right]^\top\in\mathbb{R}^{n},\quad
\mtrx{p}_{jk\mtrx{0}\mtrx{I}}:=\left[p_{jk\mtrx{0}\mtrx{e}_1},\ldots,p_{jk\mtrx{0}\mtrx{e}_n}\right]^\top\in\mathbb{R}^{n}.
\end{equation}
We denote the solutions of the eigenproblem $
\det(m\lambda^2  + c\lambda  + k)=0$ by $\lambda_{1}^x=\bar{\lambda}_{2}^x$, corresponding to the underdamped $x$-mode. We denote the solutions of the eigenproblem $\det(\lambda^2 \mtrx M + \lambda \mtrx C + \mtrx K)=0$ by $\lambda_{i}^y$, $i\in 1,\dots,2n$. We order the eigenvalues by decreasing real parts so that $\text{Re} \lambda_{2n}^y \leq \dots \leq \text{Re} \lambda_1^y< 0$. We define $\gmtrx{\Phi}_2$ as the modal matrix, which consists of the eigenvectors corresponding to each eigenvalue $\lambda_i^y$, for $i\in 1,\dots,2n$. The modeling subspace is fixed to be the slow spectral subspace, i.e., the subspace spanned by the slowest decaying mode, which, in turn, implies that $\text{Re} \lambda_{2n}^y \leq \dots \leq \text{Re}\lambda_1^y<\text{Re}\lambda_{2}^x=\text{Re}\lambda_{1}^x<0$.

\subsection{The SSM-reduced model}
Haller and Ponsioen \cite{nnmssm} give a detailed explanation of different classes of ODEs to which the theory of spectral submanifolds (SSMs) applies. The case relevant for the above setup is a time-periodic slow SSM. To discuss the conditions under which such manifolds exist, we define the absolute spectral quotient of the slow spectral subspace (aligned with the x-degree of freedom) in the form
\begin{equation}
 \Sigma := \text{Int}\frac{\text{Re} \lambda_{2n}^y}{\text{Re} \lambda_1^x},
\end{equation}
with the function $\text{Int}(\cdot)$ referring to the integer part of a real number.

Next, we assume that the low-order non-resonance condition
  \begin{equation}
    \text{Re}\lambda_{i}^y\neq  m_1 \text{Re}\lambda_{1}^x+m_2 \text{Re}\lambda_{2}^x,\quad i=1,\ldots,2n,\quad 2 \leq m_1+m_2 \leq \Sigma,
    \label{ssmnonresonance}
  \end{equation}
holds. Then a two-dimensional, analytic, invariant and attracting time-periodic SSM exists that is unique in the differentiability class $C^{\Sigma+1}$. For a more detailed formulation of absolute and relative spectral quotients of a general spectral subspace, see Haller and Ponsioen \cite{nnmssm}.

\begin{theorem}\label{thrmssm}
Under assumption (\ref{ssmnonresonance}), the single degree-of-freedom, third-order SSM-reduced model of system (\ref{xeq})-(\ref{yeq}) is of the form
\begin{equation}
\begin{split}
m \ddot{x} & +c\dot{x}+kx+p_{20\mtrx{0}\mtrx{0}}x^2+p_{11\mtrx{0}\mtrx{0}}x \dot x+p_{02\mtrx{0}\mtrx{0}} \dot x^2\\
 &+ p_{30\mtrx{0}\mtrx{0}}x^3+p_{21\mtrx{0}\mtrx{0}}x^2 \dot x+p_{12\mtrx{0}\mtrx{0}}x \dot x^2+p_{03\mtrx{0}\mtrx{0}} \dot x^3\\
 & + \left<\left(\gmtrx{\Phi}_2\mtrx{W}\right)^\top\left(x\mtrx{p}_{10\mtrx{I}\mtrx{0}}+\dot{x}\mtrx{p}_{01\mtrx{I}\mtrx{0}}\right),\mtrx{z}^{\otimes 2} \right> \\
 & + \left<\left(\gmtrx{\Phi}_2\mtrx{\tilde{W}}\right)^\top\left(x\mtrx{p}_{10\mtrx{0}\mtrx{I}}+\dot{x}\mtrx{p}_{01\mtrx{0}\mtrx{I}}\right),\mtrx{z}^{\otimes 2} \right> +\mathcal{O}(|(x,\dot x)|^4,\varepsilon|(x,\dot x)|,\varepsilon^2)=\varepsilon F_1 \sin(\phi),        \label{reducedmodelSSM}
\end{split}
\end{equation}
with $\mathbf{z}^{\otimes2}:=\left[x^2,x\dot{x},\dot{x}x,\dot{x}^2\right]^\top$ and with
\begin{equation}
\mtrx{W} = \left[\begin{array}{cccc}
\mtrx{{w}}_{11} & \mtrx{{w}}_{12} & \mtrx{{w}}_{12} & \mtrx{{w}}_{22}\\
\end{array}\right],\quad \tilde{\mtrx{W}} = \left[\begin{array}{cccc}
\mtrx{\tilde{w}}_{11} & \mtrx{\tilde{w}}_{12} & \mtrx{\tilde{w}}_{12} & \mtrx{\tilde{w}}_{22}\\
\end{array}\right],
\end{equation}
where $w_{11,i}$, $w_{12,i}$, $w_{22,i}$ and $\tilde{w}_{11,i}$, $\tilde{w}_{12,i}$, $\tilde{w}_{22,i}$ are defined in (\ref{linearsystem}) and (\ref{wtilde}), respectively.
\end{theorem}
\begin{proof}
We derive this result in Appendix \ref{proof_thrmssm}.
\end{proof}
In the present dissipative case, we could also use a higher-dimensional modeling variable, $\mtrx{x}\in\mathbb{R}^\nu$, $\nu \geq 1$. As an advantage, the corresponding SSM exists under less restrictive non-resonance conditions. By increasing the number of modeling variables, resonances between the modeling and the non-modeling modes can be avoided. We discuss the construction of the SSM-reduced model for $\mtrx{x}\in\mathbb{R}^\nu$ in Appendix \ref{ssm_construct}.

\section{Hardening or softening behavior}
\label{backbonesection}
By smoothness of all terms involved, finite Taylor expansions for the SSM-reduced dynamics necessarily converge to finite expansions of the LSM-reduced dynamics over the same modeling variable $x$, as the damping matrix $\mtrx{C}$ tends to zero. This follows because both SSM and LSM satisfy the same invariance equations and these equations depend smoothly on the dissipative perturbation to their conservative limit. As a consequence, we can determine the hardening or softening nature of a dissipative system through the periodic orbits of the conservative limit.
\subsection{Extracting the backbone curve from the LSM}
Under assumption (\ref{lsmnonresonance}), we define the backbone curve of system (\ref{systemx})-(\ref{systemy}) as the set of ordered pairs $\left(r,\omega \left(r\right)\right)$ for a periodic orbit family on the LSM, starting from the initial condition
\begin{equation}
\left(r,0,\mtrx{w}(r,0),\dot{\mtrx{w}}(r,0)\right) \in \mathcal{W}.
\end{equation}
To obtain an approximation to this curve, we first rewrite  (\ref{lsmreducednonmodal}) as
    \begin{equation}
      \ddot x + \omega_0^2 x + p^{}_{20}x^2 + (p^{}_{30}+\hat p_{30}) x^3+\hat s_{12} x \dot x^2 = 0,
      \label{perturbed}
    \end{equation}
where we introduced the short-hand notation
   \begin{gather}
   p_{20} := p_{2\mtrx{0}},\quad
   p_{30} := p_{3\mtrx{0}},\quad
  \hat{p}_{30} :=  \left<\mtrx{p}_{1\mtrx{I}},\mtrx{w}_{20}\right>, \quad
  \hat s_{12} := \left<\mtrx{p}_{1\mtrx{I}},\mtrx{w}_{02}\right>.
  \end{gather}
With the help of this notation, we obtain the following result
  \begin{theorem} \label{thm:backbone}
    The backbone curve of system (\ref{systemx})-(\ref{systemy}) can be approximated by
    \begin{equation}
    \omega(r)=\omega_0+ \omega_1 r^2+\mathcal{O}(r^3), \label{secondorderomega}
    \end{equation}
    where,
	\begin{equation}
	\omega_1=\frac{9(p^{}_{30}+\hat p_{30})\omega_0^2-10p^{2}_{20}+3\hat s_{12}\omega_0^4}{24\omega_0^3}.
	\end{equation}
  \end{theorem}

  \begin{proof}
We derive this result in Appendix \ref{app:backbone}.
  \end{proof}
We use Theorem \ref{thm:backbone} to generate frequency-energy plots for a conservative system in section \ref{oscchainexample}. Additionally, we observe that the second-order LSM approximation (\ref{secondorderomega}) takes the form of a diagonal quadratic form (see $\mtrx{w}_{11}=\mtrx 0$ in Lemma \ref{lsmcoefflemma}) for each non-modeling mode, which gives that the eigendirections are aligned with the $x$ and $\dot x$ axes. Therefore the periodic orbits on the LSM take their positional extrema in the $x$ direction for each non-modeling mode. To compute this extremum we merely need to take the values that this second order approximation takes in the first eigendirection ($x$ axis). This gives the second-order amplitude estimate for the non-modeling modes
  \begin{equation}
    r_{\eta_i}=\alpha_i r^2, \label{estimate}
  \end{equation}
along periodic orbits on the LSM. Knowing the potential for the forces in a conservative mechanical system together with the estimate (\ref{estimate}) gives a fourth-order estimate of the the energy of the periodic orbits of the LSM.

\subsection{Comparison of the backbone curves obtained from different model reduction methods}
Here we compute the backbone curves of system (\ref{modalx})-(\ref{modaly}) obtained under different reduction methods, using the approach explained in the proof of Theorem \ref{thm:backbone}. This approach leads to the general approximation formula
\begin{equation}
\omega(r)=\omega_0+ \omega_2 r^2+\mathcal{O}(r^3),
\end{equation}
where $\omega_2$ varies with the reduction method used, as summarized in Table \ref{omega2}.
\begin{table}[H]
\def\arraystretch{1.5}%
  \centering
  \begin{tabular}{|c|c|}
    \hline
    Method & $\omega_2$\\\hline
    LSM & \multirow{2}{*}{$\dfrac{9(r^{}_{3\mtrx{0}}-\left<\mtrx r_{1\mtrx{I}}, (\gmtrx{\Omega}^2\mtrx{D}_4)^{-1}\mtrx{D}_2 \mtrx s_{2\mtrx{0}}\right>)\omega_0^2-10r^{2}_{2\mtrx{0}}+6 \left<\mtrx r_{1\mtrx{I}}, (\gmtrx{\Omega}^2\mtrx{D}_4)^{-1} \mtrx s_{2\mtrx{0}}\right>\omega_0^4}{24\omega_0^3}$}\\
    Normal forms & \\ \hline
    MDs & $\dfrac{9(r_{3\mtrx{0}}-\left<\mtrx r_{1\mtrx{I}}, \gmtrx \Phi \gmtrx \Omega^{-2} \mtrx s_{2\mtrx{0}}\right>)\omega_0^2-10r^{2}_{2\mtrx{0}}}{24\omega_0^3}$ \\ \hline
  \end{tabular}
  \caption{Hardening coefficients of different model reduction methods.}
  \label{omega2}
\end{table}
Table \ref{omega2} shows that SSM theory and normal forms give the same result for $\omega_2$, while the prediction for $\omega_2$ from modal derivatives differs substantially.

\section{Numerical simulations}
We have implemented the model reduction formulas obtained in the previous sections in Julia \cite{julia}. In this section, we illustrate the power of these formulas on several examples. In each case, the full reduced and linearized models are solved numerically using the DifferentialEquations.jl \cite{diffeqjl} package.

\subsection{An analytic example}
\label{example8}

Haller and Ponsioen \cite{slowfast} gave a minimal nontrivial example of a system on which SSM reduction can be simply demonstrated. The system is of the form
\begin{align}
  \ddot x + \left(c_1+\mu_1 x^2\right)\dot x+k_1 x+axy+b x^3 =0\label{xeqexample8},\\
  \ddot y+ c_2 \dot y+ k_2 y + c x^2=0,\label{yeqexample8}
\end{align}
which is already in modal coordinates, and hence we have $\gmtrx\Phi=\mtrx I$. The $\dot x y$, $x \dot y$ and $\dot x \dot y$ terms in equation (\ref{xeqexample8}) are absent. All second-order nonlinearities in $x,\dot x $ are zero, therefore $p_{20\mtrx{0}\mtrx{0}}=p_{11\mtrx{0}\mtrx{0}}=p_{02\mtrx{0}\mtrx{0}}=0$. In addition, $p_{03\mtrx{0}\mtrx{0}}=p_{12\mtrx{0}\mtrx{0}}=0$ and $p_{21\mtrx{0}\mtrx{0}}=\mu_1$ and $p_{30\mtrx{0}\mtrx{0}}=b$. The linear part in the modeling mode $x$ is characterized by $2\zeta\omega=c_1$ and $\omega^2=k_1$.

Using equation (\ref{reducedmodelSSM}), we obtain the reduced-order model in the form
\begin{equation}
\begin{split}
  \ddot x + c_1 \dot x+ k_1 x+\left(b+ a w_{11}\right)x^3+\left(\mu_1+2a w_{12} \right) x^2 \dot x+ \left(a w_{22} \right) x \dot x^2 + \mathcal{O}(4)=0, \label{reducedexample8}
\end{split}
\end{equation}
in agreement with \cite{slowfast}. We show a simulation of the reduced dynamics on the SSM in Figure \ref{3dexamplesimulation}. Additionally, we show a trajectory of the full system, illustrating its convergence to the reduced dynamics on the SSM.

\begin{figure}[H]
\centering
\includegraphics[scale=.2]{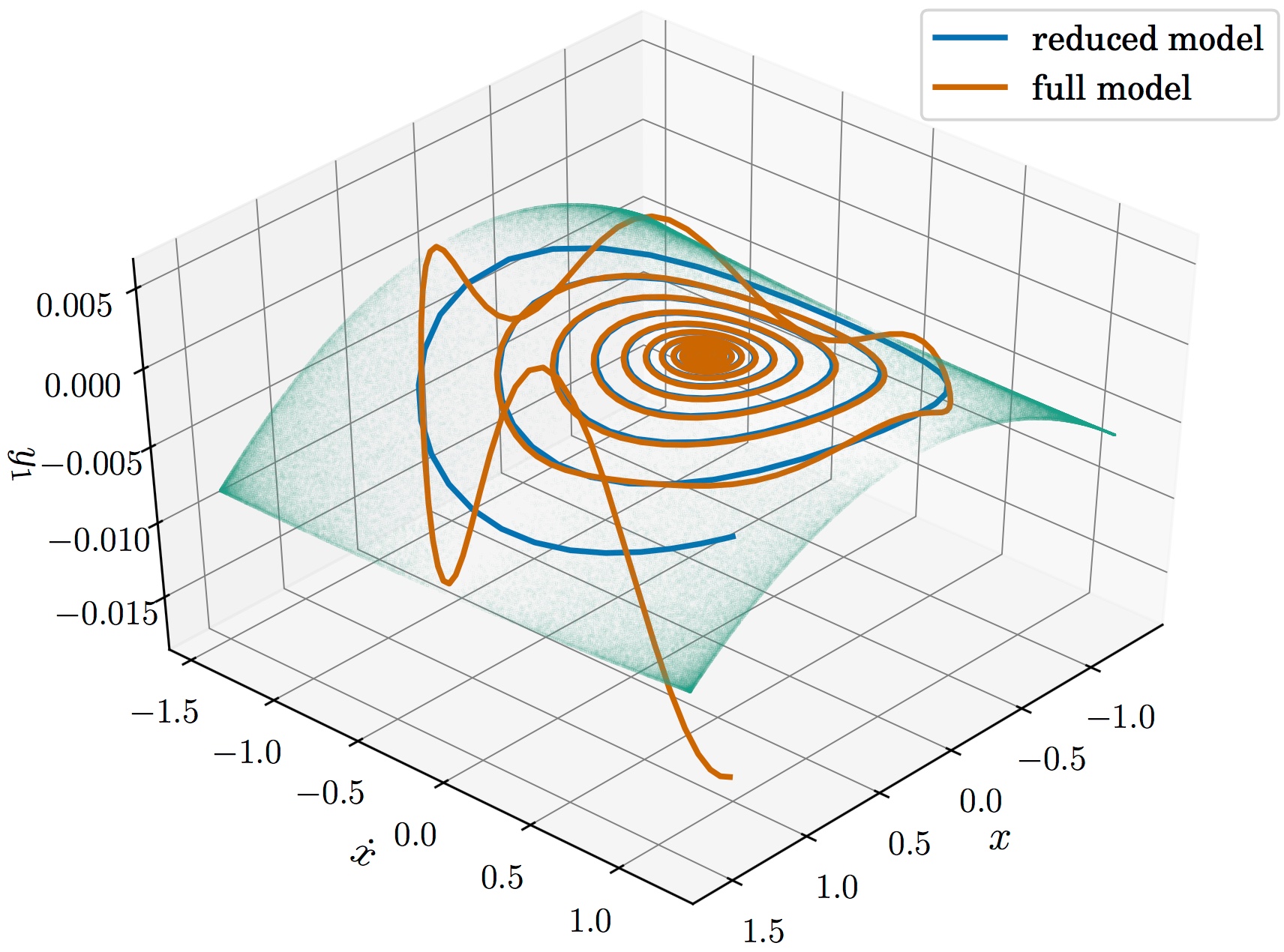}
\caption{Synchronization between near-SSM trajectories of system (\ref{xeqexample8})-(\ref{yeqexample8}) and its reduced model (\ref{reducedexample8}). The dotted green surface is the second-order approximation of the SSM. The parameter values are $c_1=0.1, \mu_1=0.02, k_1=1, a=0.06, b=0.02, c_2=1.2, k_2=27, c=0.08$.}
\label{3dexamplesimulation}
\end{figure}

All of our reduced models are defined locally around a fixed point. Their domain of validity reduces as the parameters approach a configuration in which the non-resonance conditions (\ref{ssmnonresonance}) no longer hold. In the resonant case, the reduced models are not defined anymore. Using the same example, we demonstrate the breakdown of the reduced model in a near-resonant case, where we show that the response of the reduced model is inaccurate even when the full response of the system is in the linear regime (see Figure \ref{nearresonance}).

\begin{figure}[H]
\centering
\includegraphics[scale=.65]{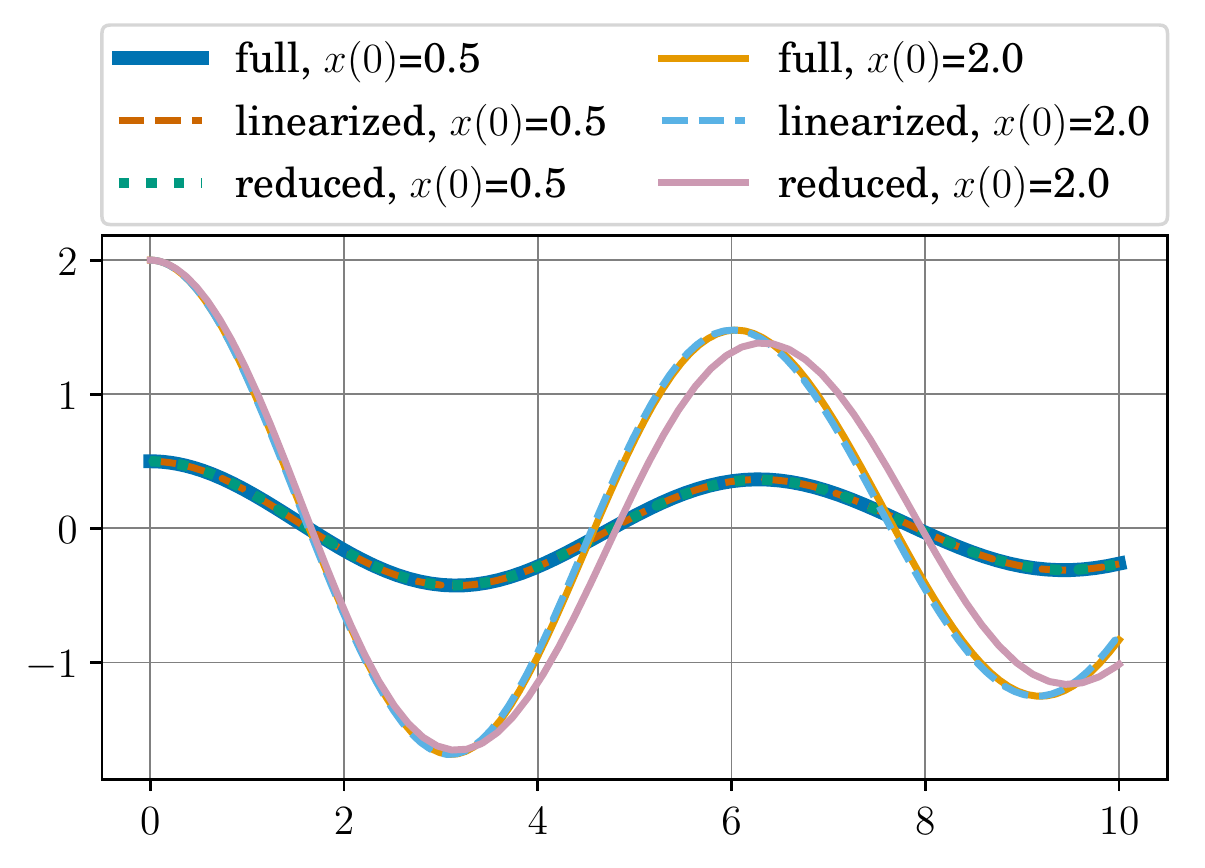}
\caption{Solutions of system (\ref{xeqexample8})-(\ref{yeqexample8}) in the near-resonant ($1:2.01$) case with a resonance parameter $r=2.01$. The parameter values are $c_1=0.1, \mu_1=0.02, k_1=1, a=0.06, b=0.02, c_2=r c_1, k_2=r^2 k_1, c=0.08$.}
\label{nearresonance}
\end{figure}

\subsection{Nonlinear oscillator chain}
\label{oscchainexample}

We continue with a twelve-degree-of-freedom, third-order nonlinear mechanical system of the form
  \begin{equation}
    \begin{split}
      m \ddot q_j+&c \left(2 \dot q_j-\dot q_{j+1}-\dot q_{j-1}\right)+k \left(2 q_j- q_{j+1}- q_{j-1}\right) \\+& \kappa_2 \left(q_j-q_{j-1}\right)^2- \kappa_2 \left(q_{j+1}-q_{j}\right)^2 \\+& \kappa_3 \left(q_j-q_{j-1}\right)^3+ \kappa_3 \left(q_j-q_{j+1}\right)^3=0,\quad j=1,\dots, 12,\quad q_0\equiv 0,\quad q_{13}\equiv 0, \label{syschain}
    \end{split}
  \end{equation}
describing the mechanical model shown in Figure \ref{oscchainmechmodel}. In this system, all the mass, spring and damping coefficients are identical, i.e., $m=k=c=1$. The springs admit a potential function of the form
\begin{equation}
V=\frac{1}{2} k (q_j-q_{j-1})^2+\frac{1}{3} \kappa_2 (q_j-q_{j-1})^3+\frac{1}{4} \kappa_3 (q_j-q_{j-1})^4.
\end{equation}

As system (\ref{syschain}) is not in the form of (\ref{xeq})-(\ref{yeq}), a linear transformation is needed first to linearly decouple the modeling mode from the non-modeling ones. Since the mass matrix is diagonal, while the stiffness and damping matrices are tridiagonal T\"oplitz matrices, we can compute the modal matrix by using the closed form formula (cf. \cite{toplitz})

\begin{equation}
  \Phi_{ij}=\sin{\frac{i j \pi}{n+1}}.
\end{equation}

For the conservative limit of system (\ref{syschain}), we compute the second-order approximated backbone curves and verify our result via numerical continuation using the \texttt{po} toolbox of \textsc{coco} \cite{coco} (see Figure \ref{oscchaibackbone}). To quantify the region of validity of these approximations, we refer to Figure \ref{oscchainfreqenergy}, where we observe an increase in energy of two orders of magnitude into the nonlinear regime for both the hardening and the softening systems.

Additionally, we numerically simulate the original damped system, the reduced system based on modal derivatives and the reduced system based on SSM theory. The mismatch between the modal derivatives and the SSM-reduced model is clearly visible in Figure \ref{12dofoscchain}.



  \begin{figure}[H]
    \centering
    \includegraphics[scale=1]{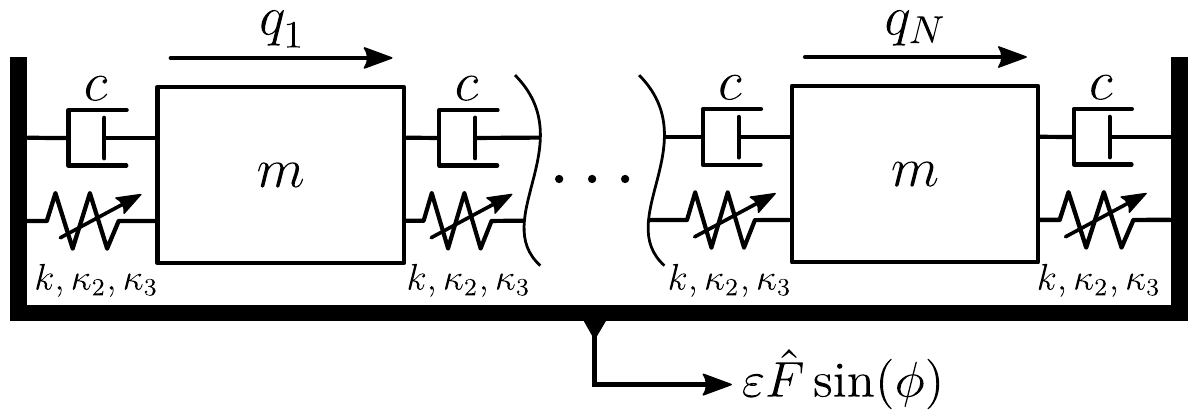}
    \caption{Illustration of the nonlinear oscillator chain described by eq. (\ref{syschain}).}
    \label{oscchainmechmodel}
  \end{figure}

    \begin{figure}[H]
      \centering
      \includegraphics[scale=.55]{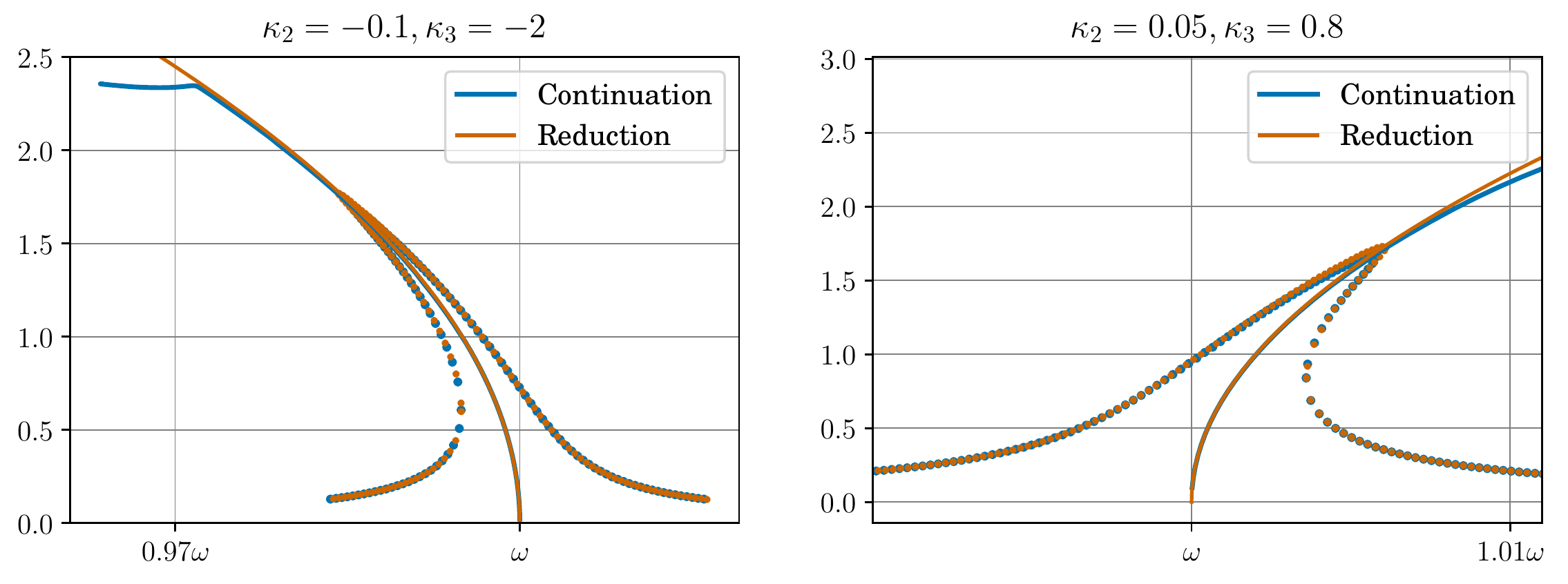}
      \caption{Conservative backbone curves and force response curves of the damped system (\ref{syschain}) obtained using \textsc{coco} and the LSM/SSM reduced model. The forcing is in both cases $\hat F \epsilon = 0.0013$ and the damping is $c=0.01$}
      \label{oscchaibackbone}
    \end{figure}

  \begin{figure}[H]
    \centering
    \includegraphics[scale=.55]{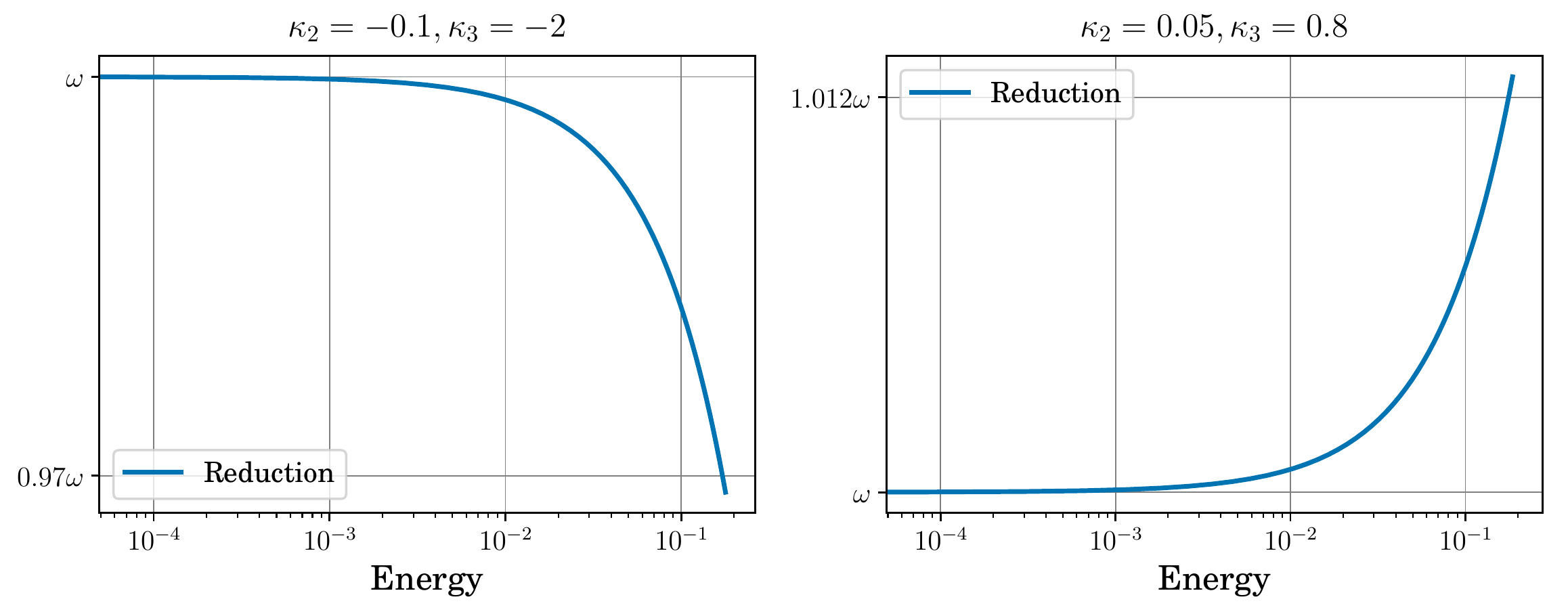}
    \caption{Frequency-energy diagrams based on the LSM reduced model for the hardening and softening example of Figure \ref{oscchaibackbone}.}
    \label{oscchainfreqenergy}
  \end{figure}

\begin{figure}[H]
  \centering
  \includegraphics[scale=.55]{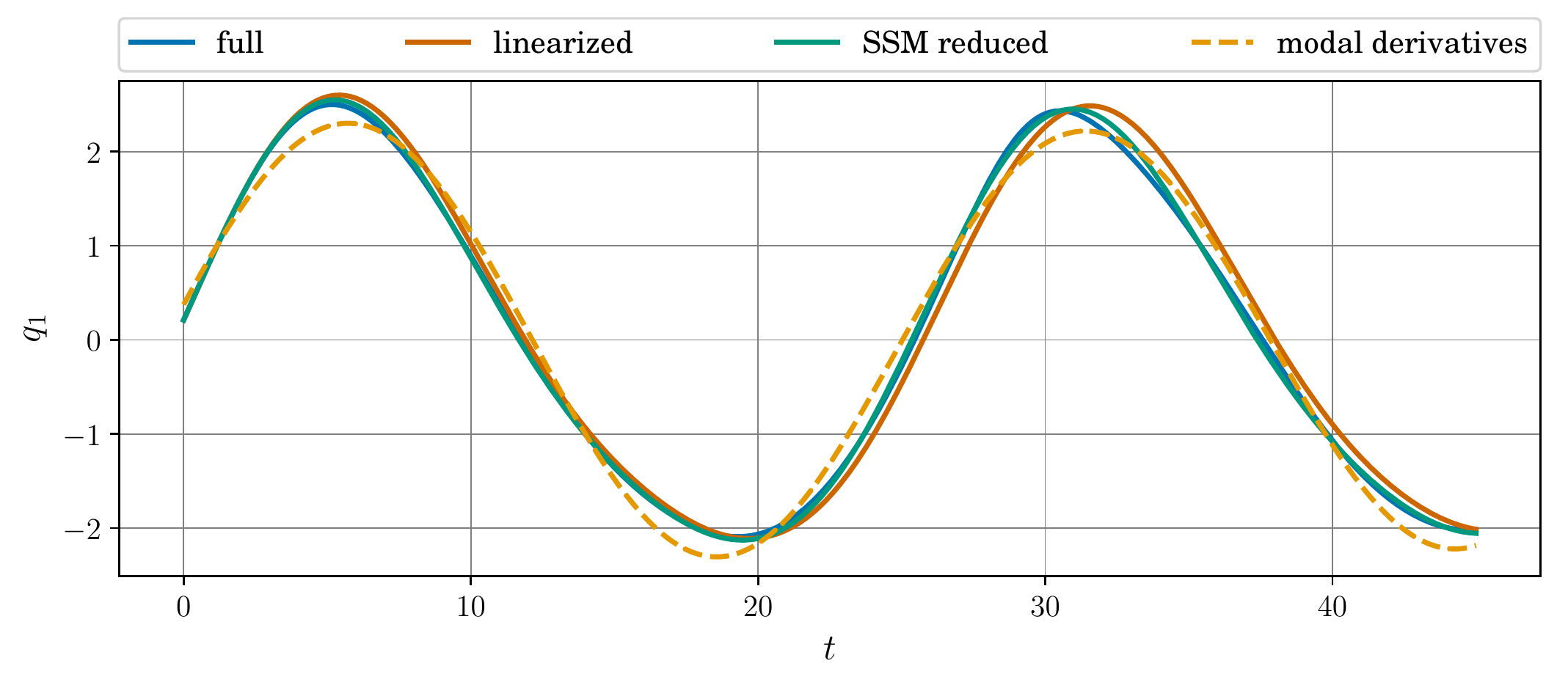}
  \caption{First coordinate against time in the numerical simulation of the 12-DoF oscillator chain (\ref{syschain}) showing the full system, the SSM-reduced, the MD-reduced and the linearized system to show the highly nonlinear character of the system. Parameters used: $c=0.05,\kappa_2=0.01 $ and  $\kappa_3=0.01$}
  \label{12dofoscchain}
\end{figure}

\subsection{Nonlinear Timoshenko beam}
  \label{examplebeam}
In this section, we construct an SSM-reduced model for a discretized nonlinear Timoshenko beam consisting of three elements, resulting in a 32-dimensional phase space (cf. Ponsioen et al. \cite{Ponsioen2018}).
The chosen beam parameter values are listed in Table \ref{tab:system_par_beam_ex}.
\begin{table}[H]
\begin{centering}
\begin{tabular}{|c|c|}
\hline
Parameter & Value\tabularnewline
\hline
\hline
$L$ & $\unit[1200]{mm}$\tabularnewline
\hline
$h$ & $\unit[40]{mm}$\tabularnewline
\hline
$b$ & $\unit[40]{mm}$\tabularnewline
\hline
$\rho$ & $\unit[7850\cdot10^{-9}]{kg\text{ }mm^{-3}}$\tabularnewline
\hline
$E$ & $\unit[90]{GPa}$\tabularnewline
\hline
$G$ & $\unit[34.6]{GPa}$\tabularnewline
\hline
$\eta$ & $\unit[13.4]{MPa\text{ }s} $\tabularnewline
\hline
$\mu$ & $\unit[8.3]{MPa\text{ }s}$\tabularnewline
\hline
\end{tabular}
\par\end{centering}
\caption{Geometric and material parameters.
 \label{tab:system_par_beam_ex}}
\end{table}
Here $L$ is the length of the beam; $h$ the height of the beam; $b$ the with of the beam; $\rho$ the density; $E$ the Young's modulus; $G$ the shear modulus; $\eta$  the axial material damping constant and $\mu$ the shear material damping constant.

For the simulation of the full stiff system, we used a fourth-order \cite{KenCarp4} ESDRIK integrator with a relative and absolute tolerance of $10^{-8}$. The reduced system was integrated with an adaptive $7/6$ Runge-Kutta scheme \cite{Vern7} where we set the tolerances equal to $10^{-12}$. With these settings a speed up factor of 2820 was obtained (from $197$s to $0.07$s) between the integration of the full system and the reduced model, see Figure \ref{fig:beam}.

\begin{figure}[H]
\centering
\includegraphics[scale=0.9]{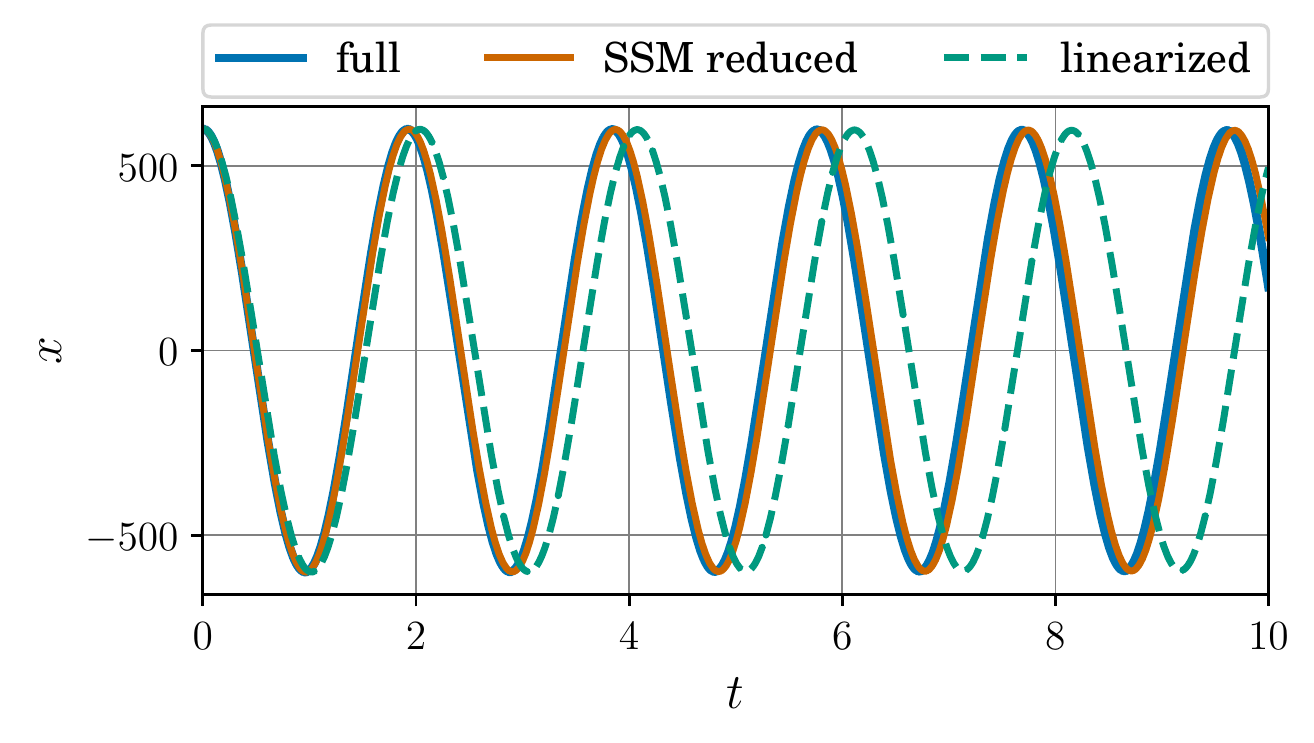}
\caption{Simulation results for the time histories of the 16 DoF nonlinear Timoshenko beam. \label{fig:beam}}
\end{figure}
In Figure \ref{fig:beam_backbone}, we compare the backbone curve extracted from our LSM reduced model with the numerically continued backbone curve using the \texttt{po} toolbox of \textsc{coco}. In this case the damping parameters have been set to zero to obtain a conservative system, required for LSM reduction and the numerical backbone curve continuation.
\begin{figure}[H]
\centering
\includegraphics[scale=0.9]{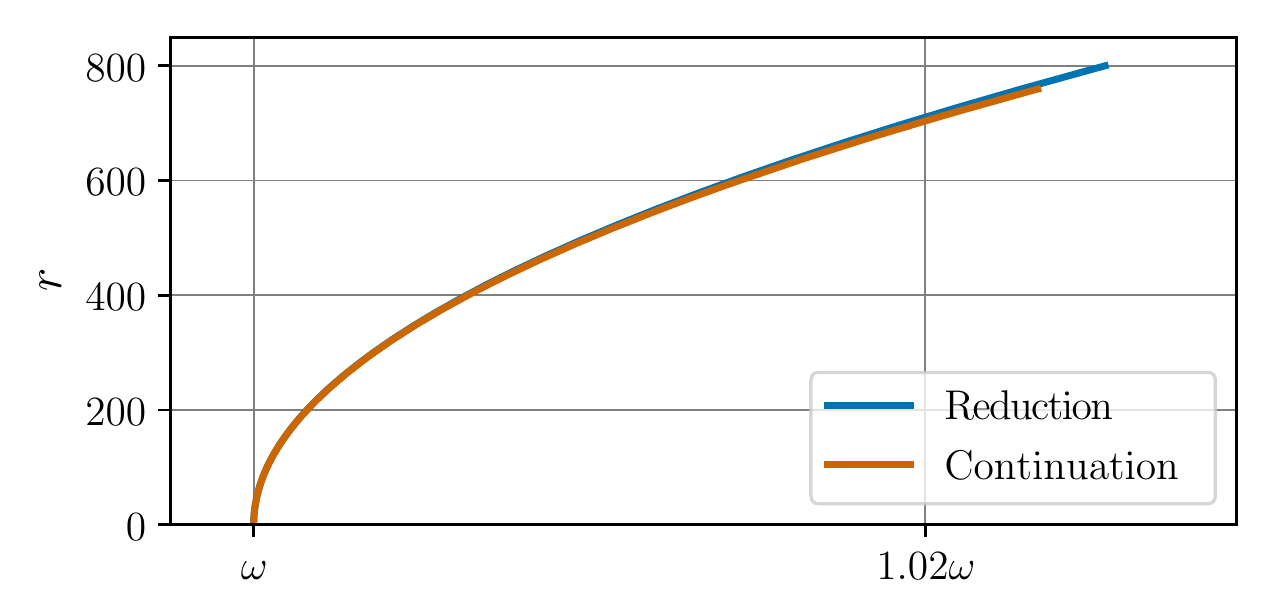}
\caption{Backbone curves of the 16-DoF nonlinear Timoshenko beam based on numerical continuation and the perturbation approach of the conservative limit detailed in section \ref{backbonesection}.\label{fig:beam_backbone}}
\end{figure}


\section{Conclusions}
We have shown that the third-order LSM- and SSM-reduced models accurately capture backbone curves and forced-response curves, respectively, of multi-degree-of-freedom mechanical systems, including a higher-dimensional nonlinear Timoshenko beam, while remaining simple to implement numerically. An advantage of our reduction method is that no near-identity transformations are needed, as opposed to normal-form-based reduction methods. The reduction method keeps the original modeling coordinate as a physically meaningful modal coordinate, resulting in a reduced model that only depends on physical and modal parameters.

Our LSM-reduction method can also be applied when the modal transformation matrix is unfeasible to compute. This is often the case for large systems in which such computations are expensive and often numerically inaccurate. In the conservative case, we have identified conditions under which the formal modal-derivatives-based reduction gives a reasonable approximation of the exact LSM reduction. Additionally, we have shown under which conditions the third-order LSM reduced model is actually a Duffing-oscillator, supporting experimental observations in the literature.

We rigorously justified that in a leading-order approximation, the SSM-reduced system of a periodically forced mechanical system can be seen as the autonomous SSM-reduced system with the modal-participation factor of the first mode added to the reduced system, which justifies the proposed normal form method by Touz{\'e} and Amabili \cite{touze2006nonlinear}. An important further development of the present results will allow for the inclusion of quasi-periodic forcing in system (\ref{systemx})-(\ref{systemy}), as the results of Haller and Ponsioen \cite{nnmssm} are general enough to allow for such forcing.

\appendix
\section{Derivation of the LSM coefficients in the general case}\label{app:der_lsm}
\begin{lemma}
  \label{lsmcoeffequationslemma}
  Under assumption (\ref{lsmnonresonance}), the unknown coefficients in (\ref{3rdapprox}) satisfy $\mtrx w_{11}=\mtrx w_{21}=\mtrx w_{03}=\mtrx 0$ and the linear equations

  \begin{equation}
    \begin{bmatrix}
      \mtrx w_{20}\\
      \mtrx w_{02}
    \end{bmatrix}
    =
    \begin{bmatrix}
      \gmtrx \Omega_\text{p}^2-2\omega^2 \mtrx I&2\omega^4 \mtrx I\\
      2 \mtrx I &\gmtrx \Omega_\text{p}^2-2\omega^2 \mtrx I\\
    \end{bmatrix}^{-1}
    \begin{bmatrix}
      -\mtrx{M}^{-1}\mtrx q_{2\mtrx{0}}\\
      \mtrx 0
      \end{bmatrix},
      \label{lsmequation2nd}
  \end{equation}

    \begin{equation}
      \begin{bmatrix}
        \mtrx w_{30}\\
        \mtrx w_{12}
      \end{bmatrix}
      =
      \begin{bmatrix}
        \gmtrx \Omega_\text{p}^2-3\omega^2 \mtrx I&2\omega^4 \mtrx I\\
        6 \mtrx I &\gmtrx \Omega_\text{p}^2-7\omega^2 \mtrx I\\
      \end{bmatrix}^{-1}
      \left(
      \begin{bmatrix}
        2{p}_{2\mtrx{0}}\mtrx{I}-\tilde{\mtrx{Q}}_{1\mtrx{I}} & -4\omega^2 p_{2\mtrx{0}} \\
        0&4p_{2\mtrx{0}}\mtrx{I}-\mtrx{M}^{-1}\tilde{\mtrx{Q}}_{1\mtrx{I}}
        \end{bmatrix}
        \begin{bmatrix}
          \mtrx{w}_{20}\\ \mtrx{w}_{02}
        \end{bmatrix}-
        \begin{bmatrix}
         \mtrx{M}^{-1} \mtrx{q}_{3\mtrx{0}}\\0
        \end{bmatrix}\right),
            \label{lsmequation3rd}
    \end{equation}
with
\begin{equation}
\gmtrx{\Omega}^2_\text{p} = {\mtrx{M}^{-1}\mtrx{K}},\quad \tilde{\mtrx{Q}}_{1\mtrx{I}}=\left[\mtrx{q}_{1\mtrx{e}_1},\ldots,\mtrx{q}_{1\mtrx{e}_n}\right]\in\mathbb{R}^{n\times n}.
\end{equation}
\end{lemma}

\begin{proof}

  \label{lsmderivation}
  In the following, we use the invariance of $\mathcal{W}$ to derive the LSM coefficients listed in lemma \ref{lsmcoefflemma}. Differentiating equation (\ref{3rdapprox}) with respect to time, we obtain, for the $i^{\text{th}}$ element of $\dot{\mtrx{y}}$,
    \begin{equation}
      \dot y_i=
      \left<\text{D}w_i,
      \begin{bmatrix}
        \dot x \\ \ddot x
      \end{bmatrix}\right> +\mathcal{O} (|(x,\dot x)|^4),
    \end{equation}
  where the gradient of $w_i$ is
      \begin{equation}
        \text{D}w_i=
        \begin{bmatrix}
          2w_{20,i} x+ w_{11,i}\dot x+3w_{30,i} x^2 + 2 w_{21,i} x \dot x+  w_{12,i} \dot x^2\\
          w_{11,i} x+2w_{02,i}  \dot x +  w_{21,i} x^2+2  w_{12,i} x \dot x+3  w_{03,i} \dot x^2
        \end{bmatrix}.
      \end{equation}
  Similarly, the second time derivative can be written as
      \begin{equation}
        \ddot y_i=
        \left<\text{D}w_i,
        \begin{bmatrix}
          \ddot x \\ \dddot x
        \end{bmatrix}\right>
        +
        \left<
        \begin{bmatrix}
          \dot x \\ \ddot x
        \end{bmatrix},\,
        \text{D}^2w_i
        \begin{bmatrix}
          \dot x \\ \ddot x
        \end{bmatrix}
        \right>+\mathcal{O} (|(x,\dot x)|^4),
        \label{ddotw}
      \end{equation}
  where the Hessian of $w_i$ is written as
      \begin{equation}
        \text{D}^2 w_i=
        \begin{bmatrix}
        2w_{20,i}+6w_{30,i} x + 2 w_{21,i} \dot x&
        w_{11,i}+2  w_{21,i} x+2  w_{12,i} \dot x \\
        w_{11,i}+2  w_{21,i} x+2  w_{12,i} \dot x &
        2w_{02,i}+2  w_{12,i} x+6  w_{03,i} \dot x
      \end{bmatrix}.
      \end{equation}

  We express the higher order derivatives, $\ddot x$ and $\dddot x$, in equation (\ref{ddotw}) as a function of $x$ and $\dot x$, by using equation (\ref{systemx}), restricted to the manifold,
  \begin{equation}
    \ddot x=-(\omega^2x+p_{2\mtrx{0}}x^2)+\mathcal{O} (|(x,\dot x)|^3).
    \label{ddotx}
  \end{equation}
  Differentiating equation (\ref{ddotx}) with respect to time yields,
  \begin{equation}
    \dddot x=-(\omega^2\dot x+2p_{2\mtrx{0}}x \dot x)+\mathcal{O} (|(x,\dot x)|^3).
    \label{dddotx}
  \end{equation}
  Substitution of equations (\ref{ddotx}) and (\ref{dddotx}) in (\ref{ddotw}) gives
    \begin{align}
      \begin{split}
          \ddot y_i=&(2w_{02,i}\omega^4-2w_{20,i} \omega^2)x^2-4w_{11,i} \omega^2x\dot x+(2w_{20,i}-2w_{02,i}\omega^2)\dot x^2\\&+(-2w_{20,i} p_{2\mtrx{0}}-3w_{30,i}\omega^2+4w_{02,i}\omega^2p_{2\mtrx{0}}+2\omega^4 w_{12,i})x^3\\&+(-5w_{11,i} p_{2\mtrx{0}}-7 w_{21,i}\omega^2+6 w_{03,i}\omega^4)x^2\dot x\\&+(6w_{30,i}-4w_{02,i} p_{2\mtrx{0}}-7 w_{12,i}\omega^2)x\dot x^2\\&+(2 w_{21,i}-3 w_{03,i}\omega^2)\dot x^3+\mathcal{O} (|(x,\dot x)|^4).
      \end{split}
      \label{lefthandside}
    \end{align}

  Rewriting equation (\ref{systemy}) and substituting $\gmtrx y = \mtrx w(x,\dot{x})$, we obtain the following expression for the $i^\text{th}$ element of $\ddot{\mtrx y}$, truncated at $\mathcal{O}(|(x,\dot x)|^4)$
   \begin{equation}
     \begin{split}
         \ddot y_i=&-\left(\left(M^{-1}K\right)_{ij} w_{20,j}+\left(M^{-1}\right)_{ij}q_{2\mtrx{0},j}\right)x^2-\left(M^{-1}K\right)_{ij} w_{11,j} x\dot x-\left(M^{-1}K\right)_{ij} w_{02,j} \dot x^2\\&-\left(\left(M^{-1}K\right)_{ij} w_{30,j}+\left(M^{-1}\right)_{ij}q_{1\mtrx{e}_k,j} w_{20,k}+ \left(M^{-1}\right)_{ij}q_{3\mtrx{0},j}\right)x^3\\& -\left(\left(M^{-1}K\right)_{ij}  w_{21,j}+\left(M^{-1}\right)_{ij}q_{1\mtrx{e}_k,j} w_{11,k}\right)x^2\dot x\\&-\left(\left(M^{-1}K\right)_{ij}  w_{12,j}+\left(M^{-1}\right)_{ij}q_{1\mtrx{e}_k,j}w_{02,k} \right) x \dot x^2-\left(M^{-1}K\right)_{ij}  w_{03,j} \dot x^3 +\mathcal{O} (|(x,\dot x)|^4).
       \end{split}
       \label{righthandside}
   \end{equation}

  Comparing the coefficients of each monomial term in (\ref{lefthandside}) and (\ref{righthandside}) leads to a system of algebraic equations for the second order coefficients

    \begin{align}
      &2\omega^4\delta_{ij}w_{02,i}+\left(\left(M^{-1}K\right)_{ij} -2 \omega^2\delta_{ij}\right)w_{20,j}=-\left(M^{-1}\right)_{ij}q_{2\mtrx{0},j}, \label{second1}\\
      &\left(4 \omega^2\delta_{ij}-\left(M^{-1}K\right)_{ij} \right) w_{11,j}=0, \label{second2}\\
      &2w_{20,i}+\left(\left(M^{-1}K\right)_{ij} -2\omega^2\delta_{ij}\right)w_{02,j}=0. \label{second3}
    \end{align}
From equation (\ref{second2}), we conclude that $w_{11,i}=0$, leading to the third order coefficient equations
    \begin{align}
      &2\omega^4 w_{12,i}+\left(\left(M^{-1}K\right)_{ij} -3\omega^2\delta_{ij}\right)w_{30,j}=n_{1,i},\label{third1}\\
      &6\omega^4 w_{03,i}+\left(\left(M^{-1}K\right)_{ij} -7\omega^2\delta_{ij}\right) w_{21,j}=0,\label{third2}\\
      &6w_{30,i}+\left(\left(M^{-1}K\right)_{ij} -7\omega^2\delta_{ij}\right) w_{12,j}=n_{3,i},\label{third3}\\
      &2 w_{21,i}+\left(\left(M^{-1}K\right)_{ij} -3\omega^2\delta_{ij}\right) w_{03,j}=0,\label{third4}
    \end{align}
  where
    \begin{align}
      n_{1,i}&=-\left(M^{-1}\right)_{ij} q_{3\mtrx{0},j}+2p_{2\mtrx{0}}w_{20,i}-\left(M^{-1}\right)_{ij}q_{1\mtrx{e}_k,j} w_{20,k}-4\omega^2 p_{2\mtrx{0}}w_{02,i},\\
      n_{3,i}&=4p_{2\mtrx{0}}w_{02,i}-\left(M^{-1}\right)_{ij}q_{1\mtrx{e}_k,j}w_{02,k}.
    \end{align}
  Equations (\ref{third2}) and (\ref{third4}) can be solved for $\mtrx{w}_{21}$ and $\mtrx{w}_{03}$ leading to $\mtrx{w}_{21}=\mtrx{w}_{03}=\mtrx 0$ Note that the algebraic equations (\ref{third1})-(\ref{third4}), related to the third order LSM coefficients, depend on the second order LSM coefficients. Therefore one has to solve for the quadratic LSM coefficients first, and substitute the result into the algebraic equations related to the third order. Note that the coupling between different modes of the system only occurs at third-order. Introducing the matrix notation $\mtrx{\Omega}^2_\text{p} := {\mtrx M^{-1}\mtrx K}$ and $\tilde{\mtrx{Q}}_{1\mtrx{I}}:=\left[\mtrx{q}_{1\mtrx{e}_1},\ldots,\mtrx{q}_{1\mtrx{e}_n}\right]$, yields the result stated in Lemma \ref{lsmcoeffequationslemma}.
\end{proof}

\section{Derivation of the LSM coefficients when the non-modeling modes are available \label{app:lsm_avai}}
\begin{lemma}
  \label{lsmcoefflemma}
  Under assumption (\ref{lsmnonresonance}), the unknown coefficients in (\ref{3rdapproxmodal}) satisfy

    \begin{align}
     \tilde{\mtrx{w}}_{20}=&-(\gmtrx{\Omega}^2\mtrx{D}_4)^{-1}\mtrx{D}_2 \mtrx s_{2\mtrx{0}},\\
     \tilde{\mtrx{w}}_{11}=&\mtrx 0,\\
     \tilde{\mtrx{w}}_{02}=&2(\gmtrx{\Omega}^2\mtrx{D}_4)^{-1} \mtrx s_{2\mtrx{0}},\\
      \begin{split}
     \tilde{\mtrx{w}}_{30}=& (\mtrx{D}_1\mtrx{D}_9)^{-1}\mtrx{D}_7\mtrx n_1
      +\frac{1}{6}\left[\mtrx I-(\mtrx{D}_1\mtrx{D}_9)^{-1}\mtrx{D}_7\mtrx{D}_3\right]\mtrx n_2,
      \end{split}\\
    \tilde{\mtrx{w}}_{21}=& \mtrx 0,\\
    \begin{split}
      \tilde{\mtrx{w}}_{12}=&   - (\mtrx{D}_1\mtrx{D}_9)^{-1}\left[6\mtrx n_1 -\mtrx{D}_3\mtrx n_2\right],
    \end{split}\\
    \tilde{\mtrx{w}}_{03}=&  \mtrx 0,
  \end{align}
	where
  \begin{align}
      \mtrx n_1&:=- \mtrx{s}_{3\mtrx{0}}+(2 r_{2\mtrx{0}}\mtrx{I}- \tilde{\mtrx{S}}_{1\mtrx{I}})\tilde{\mtrx{w}}_{20}-4\omega^2 r_{2\mtrx{0}}\tilde{\mtrx{w}}_{02},\\
      \mtrx n_2&:=(4r_{2\mtrx{0}}\mtrx I-\tilde{\mtrx{S}}_{1\mtrx{I}})\tilde{\mtrx{w}}_{02},\\
      \mtrx{D}_k&:= (\gmtrx \Omega ^2-k\omega^2 \mtrx I),\quad k\in\mathbb{Z}.
  \end{align}
\end{lemma}

\begin{proof}
  This result directly follows from transforming equations (\ref{second1}-\ref{third4}) in modal coordinates. This step makes $\mtrx M^{-1} \mtrx K$ diagonal and decouples the equations from each other for every mode:

  \begin{align}
    &2\omega^4\tilde w_{02,i}+\left(\omega_i^2 -2 \omega^2\right)\tilde w_{20,i}=-s_{20,i}, \\
    &\left(4 \omega^2-\omega_i^2 \right) \tilde w_{11,i}=0, \\
    &2\tilde w_{20,i}+\left(\omega_i^2 -2\omega^2\right)\tilde w_{02,i}=0,\\
    &2\omega^4 \tilde w_{12,i}+\left(\omega_i^2 -3\omega^2\right)\tilde w_{30,i}=n_{1,i},\\
    &6\omega^4 \tilde w_{03,i}+\left(\omega_i^2 -7\omega^2\right) \tilde w_{21,i}=0,\\
    &6\tilde w_{30,i}+\left(\omega_i^2 -7\omega^2\right) \tilde w_{12,i}=n_{3,i},\\
    &2 \tilde w_{21,i}+\left(\omega_i^2 -3\omega^2\right) \tilde w_{03,i}=0,
  \end{align}
where
  \begin{align}
    n_{1,i}&=- s_{30,i}+2r_{20}\tilde w_{20,i}-s_{11,ij} \tilde w_{20,j}-4\omega^2 r_{20}\tilde w_{02,i},\\
    n_{3,i}&=4r_{20}\tilde w_{02,i}-s_{11,ij}\tilde w_{02,j}.
  \end{align}
The solution of this system gives the statement of Lemma \ref{lsmcoefflemma}.
\end{proof}

The LSM theory does not guarantee the existence of a reduced-order model in case of resonance. The LSM coefficients will blow up exactly when the matrices $\gmtrx{\Omega}^2\mtrx{D}_4$ and $\mtrx{D}_1\mtrx{D}_9$ become singular. These matrices are diagonal, where their $i^\text{th}$ diagonal element can be written as
  \begin{align}
    &\left(\Omega^2 D_4\right)_{ii}=\omega_{i}^{2}\left(\omega_{i}-2 \omega \right)\left(\omega_{i}+2\omega \right),\\
    & \left(D_1 D_9\right)_{ii}=\left(\omega_{i}-\omega\right) \left(\omega_{i}+\omega\right) \left( \omega_{i}-3\omega\right)\left(\omega_{i}+3\omega\right).
  \end{align}
As can be seen, 1:1, 1:2 and 1:3 lower-order resonances are explicitly causing a blow-up in the expressions. Higher-order resonances will have the same effect in higher-order approximation. Even though the lower-order approximations are formally computable, there seems to be no awareness of this in existing literature.

\section{Proof of Theorem \ref{thm_ham}\label{app:ham}}
For convenience, we restate our two-dimensional first-order system (\ref{firstorder})
\begin{equation}
\begin{split}
\dot{x}=&y, \\
\dot{y}=&-\omega^2 x -\alpha x^2 -\beta x^3 - \gamma xy^2. \label{firstorder_app}
\end{split}
\end{equation}
Assuming that system (\ref{firstorder_app}) is Hamiltonian, then it must be of the general form
  \begin{equation}
        \begin{bmatrix}
        \dot{x} \\
        \dot{y}
        \end{bmatrix} = a(x,y)JDH(x,y)=
        \begin{bmatrix}
         a(x,y) \partial_y H\\
        -a(x,y) \partial_x H
        \end{bmatrix}. \label{genhamil}
  \end{equation}
  For simplicity we assume that the scalar function $a(x,y)$ only depends on $x$. Equating equations (\ref{firstorder_app}) and (\ref{genhamil}), we obtain
  \begin{align}
        \partial_y H(x,y) &= \frac{y}{a(x)}, \label{ham1}\\
        \partial_x H(x,y) &= \frac{\omega^2 x +\alpha x^2 + \beta x^3 + \gamma xy^2}{a(x)}. \label{ham2}
  \end{align}
  Integrating equation (\ref{ham1}) over $y$ yields
  \begin{equation}
  H(x,y) = \frac{y^2}{2a(x)}+F(x), \label{hamyinty}
  \end{equation}
  where $F(x)$ is a scalar function, arising from the integration, that can solely depend on $x$. Differentiating equation (\ref{hamyinty}) with respect to $x$ and equating the result with equation (\ref{ham2}), gives
  \begin{align}
  -\frac{y^2}{2a(x)^2}\partial_x a(x)+\partial_x F(x)
  = \frac{\gamma xy^2}{a(x)} + \frac{\omega^2 x +\alpha x^2 + \beta x^3}{a(x)}.
  \end{align}
  We observe that $\partial_x F(x)$ only depends on $x$, therefore we must have that
  \begin{align}
  \partial_x a(x) = -2\gamma x a(x), \quad \partial_x F(x) = \frac{\omega^2 x +\alpha x^2 + \beta x^3}{a(x)},
  \end{align}
  having the solutions
  \begin{align}
  a(x) &= e^{-\gamma x^2}, \\
  F(x) &= \frac{\omega^2}{2\gamma}e^{\gamma x^2} + \int e^{\gamma x^2}\left(\alpha x^2 + \beta x^3\right)dx.
  \end{align}
  Substituting the solutions for $a(x)$ and $F(x)$ into equation (\ref{hamyinty}), proves Theorem \ref{thm_ham}. \qed

\section{Proof of Theorem \ref{mdtheorem}\label{app:md}}
For algebraic equivalence of the MDs and the LSM reduced models, we must have that $\tilde{\mtrx{w}}_{02} =\mtrx 0$ because the MDs reduction misses the $\dot{x}^2$ term. This condition leads to the equation
    \begin{equation}
	(\gmtrx{\Omega}^2\mtrx{D}_4)^{-1} \mtrx s_{2\mtrx{0}} = \mtrx 0,
    \end{equation}
which is satisfied only in the case of $\mtrx s_{2\mtrx{0}}=\mtrx 0$, which in turn leads to flat manifolds from both the MDs and the LSM reductions.

The main statement of the theorem follows from the following limits for $\mtrx s_{2\mtrx{0}}\neq \mtrx 0$, written in coordinates:

    \begin{align}
      \begin{split}
        \lim_{\Omega_i \rightarrow \infty} \frac{\alpha_i}{\Theta_i}&= \lim_{\Omega_i \rightarrow\infty} 		\frac{\Omega_i^2-2\omega^2}{\Omega_i^2-4\omega^2}=1,\label{alphatheta}\\
        \lim_{\Omega_i \rightarrow \infty} \frac{\gamma_i}{\alpha_i}&=0.
      \end{split}
    \end{align}
Observe that in the resonant limit cases, the error of the MDs reduction is unbounded
  \begin{equation}
    \lim_{\Omega_i\rightarrow 2\omega} \frac{\alpha_i}{\Theta_i}=\infty,
  \end{equation}
as the MDs reduction method does not take the possible resonances into consideration. This concludes the proof of Theorem \ref{mdtheorem}. \qed

\section{Construction of the time-periodic SSM reduced model in the general case \label{ssm_construct}}

\subsection{Setup}
We now consider $(n+\nu)$-degree-of-freedom, periodically forced, damped nonlinear mechanical systems with a linearly independent partition of the variables so that $\mtrx x\in \mathbb{R}^\nu$ (modeling variables) and  $\mtrx y\in \mathbb{R}^n$ (non-modeling variables):
    \begin{gather}
        \mtrx M_1 \mtrx {\ddot x} + \mtrx C_1 \mtrx {\dot x}+ \mtrx K_1 \mtrx x +\mtrx P(\mtrx x,\mtrx {\dot x}, \mtrx y, \mtrx{\dot y})=\varepsilon \mtrx{F}_1\sin(\phi), \label{xeq_app}\\
        \mtrx M_2 \mtrx{\ddot y} + \mtrx C_2 \mtrx{\dot y}+ \mtrx K_2 \mtrx y +\mtrx Q(\mtrx x,\mtrx{\dot x}, \mtrx y, \mtrx{\dot y})=\varepsilon \mtrx{F}_2 \sin(\phi), \label{yeq_app}  \\
0\leq\varepsilon\ll 1,\quad \phi = \Omega t, \nonumber
    \end{gather}
where the mono-harmonic  external forcing does not depend on positions and velocities.

We assume Rayleigh-damping, which implies that the modal matrices simultaneously diagonalize the mass, the stiffness and the damping matrices. We denote the solutions of the eigenproblem $\det(\lambda^2 \mtrx M_1+\lambda\mtrx C_1+\mtrx K_1)=0$ by $\lambda_{i}^{x}$, $i\in 1,\dots,2\nu$. The eigenvalues are ordered by their real parts so that $\text{Re} \lambda_{2\nu}^x \leq \dots \leq \text{Re} \lambda_1^x< 0$.  The solutions of the eigenproblem $\det(\lambda^2 \mtrx M_2+\lambda\mtrx C_2+\mtrx K_2)=0$ are given by $\lambda_{i}^{y}$, $i\in 1,\dots,2n$. Again, we order the eigenvalues by their real parts so that $\text{Re} \lambda_{2n}^y \leq \dots \leq \text{Re} \lambda_1^y< 0$. The modeling subspace is declared to be a slow spectral subspace, i.e., the subspace spanned by the $\nu$ slowest decaying modes. This, in turn, implies that $\text{Re} \lambda_{2n}^y \leq \dots \leq \text{Re} \lambda_1^y<\text{Re} \lambda_{2\nu}^x \leq \dots \leq \text{Re} \lambda_1^x< 0$.

 \subsection{Existence of the reduced model}
In the current setting, we seek a time-periodic slow SSM. We assume that the low-order non-resonance conditions
  \begin{equation}
   \text{Re}\lambda_{j}^y\neq \sum_{i=1}^\nu \nu_i \text{Re}\lambda_{i}^x,
    \label{ssmnonresonance2}
  \end{equation}
  hold for every $j$ and nonnegative set of integers $\nu_i$ such that $2 \leq \sum_{k=1}^\nu \nu_k \leq \Sigma$. Then a $2\nu$-dimensional, analytic, invariant and attracting time-periodic SSM exists that is unique in the differentiability class $C^{\Sigma+1}$, and higher, up to analytic.

\subsection{Construction of the time-periodic SSM reduced model}
We express the damped mechanical system (\ref{xeq_app})-(\ref{yeq_app}) in modal coordinates by substituting $\mtrx x=\gmtrx \Phi_1 \gmtrx{\xi}$, $\mtrx y=\gmtrx \Phi_2 \gmtrx \eta$. Subsequently, we project the equations of motion onto the modal directions by left multiplying equations (\ref{xeq_app}) and (\ref{yeq_app}) with the transposed modal matrices $\gmtrx\Phi_1^\top$ and $\gmtrx\Phi_2^\top$:
  \begin{align}
     & \gmtrx{\ddot\xi} + \mtrx{\hat C}_1\gmtrx{\dot\xi}+\mtrx{\hat K}_1\gmtrx{\xi} + \mtrx{R} \left(\gmtrx{\xi},\gmtrx{\dot\xi}, \gmtrx \eta, \gmtrx{\dot \eta} \right)=\varepsilon \hat{\mtrx{F}}_1\sin(\phi)\text{, }\label{xieq}\\
    &\gmtrx{\ddot \eta} + \mtrx{\hat C}_2\gmtrx{\dot\eta}+\mtrx{\hat K}_2\gmtrx{\eta} + \mtrx{S} \left(\gmtrx{\xi},\gmtrx{\dot\xi}, \gmtrx \eta, \gmtrx{\dot \eta} \right)=\varepsilon \hat{\mtrx{F}}_2\sin(\phi)\text{, } \label{etaeq}
  \end{align}
Using the results from Haller and Ponsioen \cite{nnmssm}, we make use the fact that the SSM perturbs smoothly from the modeling subspace of the linear unperturbed part of system (\ref{xieq})-(\ref{etaeq}) under the addition of $\mathcal{O}(\varepsilon)$ terms in system (\ref{xieq})-(\ref{etaeq}). We restrict ourselves to the setting of Breunung and Haller \cite{thomas} and expand each non-modeling variables $\eta_i$ in $\varepsilon$. Subsequently, for different orders of $\varepsilon$, every non-modeling variable $\eta_i$ is expanded in $\gmtrx{\xi}$ and $\dot{\gmtrx{\xi}}$. After a truncating at $\mathcal{O}(|(\gmtrx{\xi},\gmtrx{\dot\xi})|^3,\varepsilon|(\gmtrx{\xi},\gmtrx{\dot\xi})|,\varepsilon^2)$, which is justified if $\gmtrx{\xi}$ and $\dot{\gmtrx{\xi}}$ are of $\mathcal{O}(\varepsilon^{\frac{1}{4}})$,  we obtain
    \begin{equation}
        \eta_i=
        \left<
        \begin{bmatrix}
          \gmtrx{\xi} \\ \gmtrx{\dot\xi}
        \end{bmatrix},
        \begin{bmatrix}
          \mtrx W_{11,i}& \mtrx W_{12,i} \\
          \mtrx W_{12,i}^\top& \mtrx W_{22,i}
        \end{bmatrix}
        \begin{bmatrix}
          \gmtrx{\xi}\\ \gmtrx{\dot\xi}
        \end{bmatrix}
        \right>
        +
        \varepsilon \bar{W}_i(\phi)
        +\mathcal{O}(|(\gmtrx{\xi},\gmtrx{\dot\xi})|^3,\varepsilon|(\gmtrx{\xi},\gmtrx{\dot\xi})|,\varepsilon^2).
        \label{SSM}
    \end{equation}
Here, $\mtrx W_{11,i}, \mtrx W_{12,i}$ and $\mtrx W_{22,i}$ are $m\times m$ matrices. We specifically aim for a third-order reduced model, which implies that the manifold only has to be approximated up to second order. We uniquely define the quadratic form in equation (\ref{SSM}), by requiring that $\mtrx W_{11,i}=\mtrx W_{11,i}^\top$ and $\mtrx W_{22,i}=\mtrx W_{22,i}^\top$.


\begin{lemma}\label{lemmassm}
  The coefficient matrices $\mtrx W_{11,i}$, $\mtrx W_{12,i}$ and $\mtrx W_{22,i}$ of the second-order SSM approximation (\ref{SSM}) solve the linear matrix equations
\begin{align}
  \begin{split}
  \mtrx S_{11,i}&= \rm{Sym}\left[\left(4\zeta_i \omega_i\mtrx{\hat K}_1 - 2\mtrx{\hat K}_1\mtrx{\hat C}_1 \right) \mtrx W_{12,i}^\top \right.\\
   & \left. + \left(2\mtrx{\hat K}_1-\omega_i^2 \mtrx I \right) \mtrx W_{11,i} -  2\mtrx{\hat K}_1 \mtrx W_{22,i} \mtrx{\hat K}_1\right],
   \end{split}\label{ssmeq1}\\
     \begin{split}
      \mtrx S_{12,i}&=-\left(\mtrx{\hat K}_1 \mtrx{\hat C}_1 - 2\zeta_i \omega_i \mtrx{\hat K}_1\right)  \mtrx W_{22,i} -(\omega_i^2\mtrx I-\mtrx{\hat K}_1) \mtrx W_{12,i}\\
      & +2\mtrx{\hat K}_1 \mtrx W_{12,i}^\top - 2\mtrx{\hat K}_1 \mtrx W_{22,i} \mtrx{\hat C}_1 -\mtrx W_{12,i}\left(\mtrx{\hat C}_1^2-\mtrx{\hat K}_1\right)\\
      & + \mtrx W_{11,i}\mtrx{\hat C}_1 + 2\zeta_i \omega_i \mtrx W_{12,i} \mtrx{\hat C}_1 - 2\zeta_i \omega_i \mtrx W_{11,i},
   \end{split}\\
  \begin{split}
  \mtrx S_{22,i}&=\rm{Sym}\left[\left(2\mtrx{\hat C}_1 -4\zeta_i \omega_i \mtrx I\right) \mtrx W_{12,i}\right.\\&+\left(4\zeta_i \omega_i \mtrx{\hat C}_1-\omega_i^2 \mtrx I-2\left(\mtrx{\hat C}_1^2-\mtrx{\hat K}_1\right)\right) \mtrx W_{22,i}\\
  & \left.+4\mtrx{\hat C}_1 \mtrx W_{12,i}^\top -2\mtrx{\hat C}_1 \mtrx W_{22,i} \mtrx{\hat C}_1-2\mtrx W_{11,i}\right]. \label{ssmeq3}
   \end{split}
\end{align}
where  $\mtrx S_{11,i}$, $\mtrx S_{12,i}$ and $\mtrx S_{22,i}$ are the $m\times m$ matrices filled with the coefficients of the $ \gmtrx{\xi},\gmtrx{\dot\xi}$ dependent terms from the quadratic part of $S_{i}( \gmtrx{\xi},\gmtrx{\dot\xi}, \gmtrx\eta, \dot {\gmtrx\eta})$.
\end{lemma}
\begin{lemma}\label{lemmassm_periodic}
The time-dependent coefficients $\bar{W}_i(\phi)$, of the second order SSM approximation (\ref{SSM}), are the periodic solutions of the second-order non-homogeneous ordinary differential equations
\begin{equation}
\Omega^2\bar{W}_i^{\prime\prime}+2\zeta_i\omega_i\Omega \bar{W}_i^\prime + \omega_i^2 \bar{W}_i = \hat{F}_{2,i}\sin(\phi),
\end{equation}
and can be expressed as
\begin{align}
\bar{W}_i(\phi) &= \frac{\omega_i^2-\Omega^2}{(\omega_i^2-\Omega^2)^2+(2\zeta_i\omega_i\Omega)^2}\hat{F}_{2,i}\sin(\phi) \\
&+ \left(\frac{(\omega_i^2-\Omega^2)^2}{2\zeta_i\omega_i\Omega((\omega_i^2-\Omega^2)^2+(2\zeta_i\omega_i\Omega)^2)}-\frac{1}{2\zeta_i\omega_i\Omega}\right)\hat{F}_{2,i}\cos(\phi).\nonumber
\end{align}
\end{lemma}
\begin{proof}
We prove Lemmas \ref{lemmassm} and \ref{lemmassm_periodic} simultaneously using a direct invariance computation. Differentiating equation (\ref{SSM}) with respect to time twice yields
    \begin{align}
      \ddot \eta_i=
      &2\left(\left<
      \begin{bmatrix}
          \gmtrx{\ddot\xi} \\ \gmtrx{\dddot  \xi}
        \end{bmatrix},
        \begin{bmatrix}
          \mtrx W_{11,i}& \mtrx W_{12,i} \\
          \mtrx W_{12,i}^\top& \mtrx W_{22,i}
        \end{bmatrix}
        \begin{bmatrix}
           \gmtrx{\xi} \\ \gmtrx{\dot\xi}
        \end{bmatrix}
        \right>
        +
        \left<
        \begin{bmatrix}
          \gmtrx{\dot\xi} \\   \gmtrx{\ddot\xi}
          \end{bmatrix},
          \begin{bmatrix}
            \mtrx W_{11,i}& \mtrx W_{12,i} \\
            \mtrx W_{12,i}^\top& \mtrx W_{22,i}
          \end{bmatrix}
          \begin{bmatrix}
            \gmtrx{\dot\xi}\\   \gmtrx{\ddot\xi}
          \end{bmatrix}
        \right>\right) \\
		& +\varepsilon \partial^2_\phi\bar{W}_i(\phi)\Omega^2
        +\mathcal{O}(|(\gmtrx{\xi},\gmtrx{\dot\xi})|^3,\varepsilon|(\gmtrx{\xi},\gmtrx{\dot\xi})|,\varepsilon^2).
        \label{ddotSSM}
      \end{align}
We express the higher derivatives $\ddot{\gmtrx{\xi}}$ and $\gmtrx{\dddot{\xi}}$ using equation (\ref{xieq}) and its derivative with respect to time

      \begin{align}
           \ddot{\gmtrx{\xi}}&= - \mtrx{\hat C}_1 \gmtrx{\dot\xi} - \mtrx{\hat K}_1  \gmtrx{\xi}+\varepsilon \hat{\mtrx{F}}_1\sin(\phi) +\mathcal{O}(|(\gmtrx{\xi},\gmtrx{\dot\xi})|^2) \text{,} \label{ddx}\\
      \begin{split}
          \gmtrx{\dddot  \xi}&=
            \left( \mtrx{\hat C}_1^2-\mtrx{\hat K}_1 \right)\gmtrx{\dot\xi} +\mtrx{\hat C}_1 \mtrx{\hat K}_1  \gmtrx{\xi}+\varepsilon \Omega \hat{\mtrx{F}}_1\cos(\phi) +\mathcal{O}(|(\gmtrx{\xi},\gmtrx{\dot\xi})|^2).
      \end{split}
      \label{dddx}
    \end{align}
Substituting expressions (\ref{ddx}) and (\ref{dddx}) into equation (\ref{ddotSSM}) gives
    \begin{equation}
  \begin{split}
      \ddot \eta_i=&
      2\left<
      \begin{bmatrix}
        -\mtrx{\hat C}_1 \gmtrx{\dot\xi} - \mtrx{\hat K}_1  \gmtrx{\xi}  \\ \left( \mtrx{\hat C}_1^2 -\mtrx{\hat K}_1 \right)\gmtrx{\dot\xi} +\mtrx{\hat C}_1 \mtrx{\hat K}_1  \gmtrx{\xi}
        \end{bmatrix},
        \begin{bmatrix}
          \mtrx W_{11,i}& \mtrx W_{12,i} \\
          \mtrx W_{12,i}^\top& \mtrx W_{22,i}
        \end{bmatrix}
        \begin{bmatrix}
           \gmtrx{\xi}\\ \gmtrx{\dot\xi}
        \end{bmatrix}
        \right>
        \\+&
        2\left<
        \begin{bmatrix}
          \gmtrx{\dot\xi} \\ -\mtrx{\hat C}_1 \dot{\gmtrx{\xi}} -\mtrx{\hat K}_1  \gmtrx{\xi}
          \end{bmatrix},
          \begin{bmatrix}
            \mtrx W_{11,i}& \mtrx W_{12,i} \\
            \mtrx W_{12,i}^\top& \mtrx W_{22,i}
          \end{bmatrix}
          \begin{bmatrix}
            \gmtrx{\dot\xi} \\ -\mtrx{\hat C}_1 \dot{\gmtrx{\xi}} - \mtrx{\hat K}_1  \gmtrx{\xi}
          \end{bmatrix}
          \right>
          \\+& \varepsilon \partial^2_\phi\bar{W}_i(\phi)\Omega^2 +\mathcal{O}(|(\gmtrx{\xi},\gmtrx{\dot\xi})|^3,\varepsilon|(\gmtrx{\xi},\gmtrx{\dot\xi})|,\varepsilon^2).
        \end{split}\label{etaddsub}
        \end{equation}
Equation (\ref{etaddsub}) is still quadratic in $ \gmtrx{\xi}$ and $\gmtrx{\dot\xi}$, and can be rearranged as
    \begin{equation}
      \ddot \eta_i=\left<
      \begin{bmatrix}
         \gmtrx{\xi}\\\gmtrx{\dot\xi}
      \end{bmatrix}
       ,
       \underbrace{
       \begin{bmatrix}
         \mtrx B_{11,i} & \mtrx B_{12,i} \\ \mtrx B_{21,i} & \mtrx B_{22,i}
       \end{bmatrix}}_{\mtrx B,i}
       \begin{bmatrix}
          \gmtrx{\xi}\\\gmtrx{\dot\xi}
       \end{bmatrix}
       \right>+\varepsilon \partial^2_\phi\bar{W}_i(\phi)\Omega^2 +\mathcal{O}(|(\gmtrx{\xi},\gmtrx{\dot\xi})|^3,\varepsilon|(\gmtrx{\xi},\gmtrx{\dot\xi})|,\varepsilon^2), \label{Bform}
    \end{equation}
where the $\mtrx B_{kl,i}$ matrices are equal to
    \begin{align}
      \mtrx B_{11,i}&=2 \mtrx{\hat K}_1 \mtrx{\hat C}_1 \mtrx W_{12,i}^\top  - 2\mtrx{\hat K}_1  \mtrx W_{11,i}+  2\mtrx{\hat K}_1  \mtrx W_{22,i} \mtrx{\hat K}_1,  \\
      \mtrx B_{12,i}&=2 \mtrx{\hat K}_1 \mtrx{\hat C}_1 \mtrx W_{22,i}-2\mtrx{\hat K}_1  \mtrx W_{12,i}-2\mtrx{\hat K}_1  \mtrx W_{12,i}^\top,\\
      \begin{split}
        \mtrx B_{21,i}&=2\left(\mtrx{\hat C}_1 ^2-\mtrx{\hat K}_1 \right)\mtrx W_{12,i}^\top-2\mtrx{\hat C}_1  \mtrx W_{11,i}+4\mtrx{\hat C}_1  \mtrx W_{22,i}\mtrx{\hat K}_1\\  &-2\mtrx W_{12,i} \mtrx{\hat K}_1,
      \end{split}\\
      \begin{split}
        \mtrx B_{22,i}&=2\left(\mtrx{\hat C}_1 ^2-\mtrx{\hat K}_1 \right)\mtrx W_{22,i}-2\mtrx{\hat C}_1  \mtrx W_{12,i}+2\mtrx{\hat C}_1  \mtrx W_{22,i} \mtrx{\hat C}_1\\ &-4\mtrx{\hat C}_1  \mtrx W_{12,i}^\top+2\mtrx W_{11,i}.
      \end{split}
    \end{align}
The $i^\text{th}$ element of equation (\ref{etaeq}) can be written as
  \begin{equation}
    \ddot \eta_i + 2\zeta_i \omega_i \dot \eta_i+ \omega_i^2 \eta_i + S_{i}( \gmtrx{\xi},\gmtrx{\dot\xi}, \gmtrx{\eta}, \gmtrx{\dot{\eta}})=\varepsilon \hat{F}_{2,i}\sin(\phi),
  \end{equation}
  which, truncated at third-order in $ \gmtrx{\xi}$ and $\gmtrx{\dot\xi}$ and taking into consideration that $\eta_i=\mathcal{O}(|(\gmtrx{\xi},\gmtrx{\dot\xi})|^2)$, leads to
  \begin{equation}
    \ddot \eta_i + 2\zeta_i \omega_i \dot \eta_i+ \omega_i^2 \eta_i +
    \left<
    \begin{bmatrix}
       \gmtrx{\xi}\\ \gmtrx{\dot\xi}
    \end{bmatrix} ,
    \begin{bmatrix}
      \mtrx S_{11,i} & \mtrx S_{12,i}\\ \mtrx S_{12,i}^\top & \mtrx S_{22,i}
    \end{bmatrix}
    \begin{bmatrix}
       \gmtrx{\xi}\\ \gmtrx{\dot\xi}
    \end{bmatrix}
    \right>+\mathcal{O}(|(\gmtrx{\xi},\gmtrx{\dot\xi})|^3)=\varepsilon \hat{F}_{2,i}\sin(\phi), \label{etaddsubr}
  \end{equation}
where $\mtrx S_{11,i}$, $\mtrx S_{12,i}$ and $\mtrx S_{22,i}$ are $\nu\times \nu$ matrices filled with coefficients of the $ \gmtrx{\xi},\gmtrx{\dot\xi}$ dependent terms from the quadratic part of $S_{i}( \gmtrx{\xi},\gmtrx{\dot\xi}, \gmtrx\eta, \dot {\gmtrx\eta})$. Substituting $\eta_i$ from equation (\ref{SSM}) and its time derivative $\dot \eta_i$ into equation (\ref{etaddsubr}), leads, after some rearrangement, to the quadratic approximation:
  \begin{equation}
    \ddot \eta_i =
      \left<
      \begin{bmatrix}
         \gmtrx{\xi} \\ \gmtrx{\dot\xi}
        \end{bmatrix},
          \mtrx C_i
        \begin{bmatrix}
           \gmtrx{\xi}\\ \gmtrx{\dot\xi}
        \end{bmatrix}
        \right>
        + \varepsilon\left(\hat{F}_{2,i}\sin{\phi}-2\zeta_i\omega_i\partial_\phi\bar{W}_i\Omega-\omega_i^2\bar{W}_i\right)
         +\mathcal{O}(|(\gmtrx{\xi},\gmtrx{\dot\xi})|^3,\varepsilon|(\gmtrx{\xi},\gmtrx{\dot\xi})|,\varepsilon^2), \label{Cform}
      \end{equation}
  where $\mtrx C_i$ denotes the hypermatrix
  \begin{align}
      \mtrx C_i&=
        \begin{bmatrix}
          C_{11,i} &C_{12,i} \\
          C_{21,i}& C_{22,i}
        \end{bmatrix},
      \\
      C_{11,i}&=-\omega_i^2 \mtrx W_{11,i}+ 4\zeta_i \omega_i \mtrx{\hat K}_1 \mtrx W_{12,i}^\top -\mtrx S_{11,i},
      \\
      C_{12,i}&=-\omega_i^2 \mtrx W_{12,i}-\mtrx S_{12,i}+4\zeta_i \omega_i  \mtrx{\hat K}_1 \mtrx W_{22,i},
      \\
      C_{21,i}&=-\omega_i^2 \mtrx W_{12,i}^\top -\mtrx S_{12,i}^\top-4\zeta_i \omega_i \mtrx W_{11,i}+ 4\zeta_i \omega_i \mtrx{\hat C}_1 \mtrx W_{12,i}^\top,
      \\
      C_{22,i}&=-\omega_i^2 \mtrx W_{22,i}-\mtrx S_{22,i}-4\zeta_i \omega_i \mtrx W_{12,i}+ 4\zeta_i \omega_i \mtrx{\hat C}_1 \mtrx W_{22,i}.
  \end{align}

At leading order, we are now able to equate the quadratic forms in equations (\ref{Bform}) and (\ref{Cform}), which implies that $\mtrx B_i$ and $\mtrx C_i$ must have equal symmetric parts, i.e.,

    \begin{equation}
      \text{Sym } \mtrx B_i =\text{Sym } \mtrx C_i,
      \label{quadeq}
    \end{equation}
which returns the linear system stated in Lemma \ref{lemmassm}. Equating the $\mathcal{O}(\varepsilon|(\gmtrx{\xi},\gmtrx{\dot\xi})|^0)$ terms will return a non-homogeneous second-order ordinary differential equation
\begin{equation}
\Omega^2\bar{W}_i^{\prime\prime}+2\zeta_i\omega_i\Omega \bar{W}_i^\prime + \omega_i^2 \bar{W}_i = \hat{F}_{2,i}\sin(\phi), \label{eq:2nd_ode}
\end{equation}
where we introduced the notation $\partial_\phi(\cdot)=(\cdot)^\prime$. Using the method of undetermined coefficients, we obtain the periodic solution
\begin{align}
\bar{W}_i(\phi) &= \frac{\omega_i^2-\Omega^2}{(\omega_i^2-\Omega^2)^2+(2\zeta_i\omega_i\Omega)^2}\hat{F}_{2,i}\sin(\phi) \\
&+ \left(\frac{(\omega_i^2-\Omega^2)^2}{2\zeta_i\omega_i\Omega((\omega_i^2-\Omega^2)^2+(2\zeta_i\omega_i\Omega)^2)}-\frac{1}{2\zeta_i\omega_i\Omega}\right)\hat{F}_{2,i}\cos(\phi),\nonumber
\end{align}
which proves Lemma \ref{lemmassm_periodic}.
\end{proof}

\section{Proof of Theorem \ref{thrmssm} \label{proof_thrmssm}}
A general solution of equations (\ref{ssmeq1})-(\ref{ssmeq3}), without specifying the dimension $\nu$, cannot be obtained. However, the solution for the practically most relevant case of reducing the system to a single degree of freedom ($\nu=1$) can be written in closed form,

  \begin{equation}
    \begin{split}
    &\begin{bmatrix}
      w_{11,i}\\w_{12,i}\\w_{22,i}
    \end{bmatrix} = \mtrx L^{-1}_i
    \begin{bmatrix}
    s_{20\mtrx{0}\mtrx{0},i}\\s_{11\mtrx{0}\mtrx{0},i}\\s_{02\mtrx{0}\mtrx{0},i}
    \end{bmatrix}, \quad
     \mtrx{L}_i = \begin{bmatrix}
      l_{11} & l_{12} & l_{13}\\
      l_{21}  & l_{22} & l_{23}\\
      l_{31} & l_{32} & l_{33}\\
    \end{bmatrix}
    ,\\
    &l_{11}=2\omega^2 - \omega_i^2
    ,\\
    &l_{12}=4\omega^2(\zeta_i\omega_i - \zeta\omega)
    ,\\
    &l_{13}=-2\omega^4
    ,\\
    &l_{21}=2(\zeta\omega - \zeta_i\omega_i)
    ,\\
    &l_{22}=4\zeta_i\omega_i\zeta\omega-4\omega^2(\zeta^2-1)-\omega_i^2
    ,\\
    &l_{23}=2\zeta_i\omega_i \omega^2 - 6\zeta\omega^3
    ,\\
    &l_{31}=-2
    ,\\
    &l_{32}=4(3\zeta\omega-\zeta_i\omega_i)
    ,\\
    &l_{33}= 2\omega^2 - \omega_i^2 + 8\zeta_i\omega_i \zeta\omega - 16 (\zeta\omega)^2.
    \end{split}
    \label{linearsystem}
  \end{equation}
The determinant of $\mtrx L_i$ can be expressed as
\begin{equation}
  \begin{split}
    D:=-\left(4 \zeta ^2 \omega ^2-4 \zeta  \zeta_i \omega  \omega_i+\omega_i^2\right) \left(8 \omega ^2 \omega_i^2 \left(2 \zeta ^2+2 \zeta_i^2-1\right)\right.\\\left. -32 \zeta  \zeta_i \omega ^3 \omega_i-8 \zeta  \zeta_i \omega  \omega_i^3+16 \omega ^4+\omega_i^4\right).
  \end{split}
\end{equation}
The construction of the SSM breaks down exactly when $\mtrx L_i$ becomes singular. This implies that the determinant $D$ is zero, which will be the case for a 1:2 resonance ($\zeta_i=\zeta$, $\omega_i=2\omega$). Note that a 1:1 resonance is allowed by the SSM theory.

The modal non-modeling coordinate $\eta_{i}$ (\ref{SSM}), for $\nu=1$, can be written as
\begin{equation}
\begin{split}
\eta_{i}=w_{i}(x,\dot{x},\phi)=&\left<\begin{bmatrix}
x\\
\dot{x}
\end{bmatrix},\begin{bmatrix}
w_{11,i} & w_{12,i}\\
w_{12,i} & w_{22,i}
\end{bmatrix}
\begin{bmatrix}
x\\
\dot{x}
\end{bmatrix}\right>
+\varepsilon \bar{W}_i(\phi)
+\mathcal{O}(|(x,\dot{x})|^3,\varepsilon|(x,\dot{x})|,\varepsilon^2)\label{etassm}
\end{split}
\end{equation}
Differentiating equation (\ref{etassm}) with respect to time and plugging in $\gmtrx{\ddot{\xi}}=\ddot x$ from equation (\ref{ddx}) gives
\begin{equation}
\begin{split}
\dot{\eta_{i}}=\tilde{w}_{i}(x,\dot{x},\phi)
=&\left< \begin{bmatrix}
x\\
\dot{x}
\end{bmatrix},\begin{bmatrix}
\tilde{w}_{11,i} & \tilde{w}_{12,i}\\
\tilde{w}_{12,i} & \tilde{w}_{22,i}
\end{bmatrix}\begin{bmatrix}
x\\
\dot{x}
\end{bmatrix}\right> +\varepsilon \partial_\phi\bar{W}_i(\phi)\Omega +\mathcal{O}(|(x,\dot x)|^3,\varepsilon|(x,\dot{x})|,\varepsilon^2),  \label{eta_dotssm}
\end{split}
\end{equation}
with
\begin{align}
&\tilde{w}_{11,i} = -2\omega^{2}w_{12,i}, \nonumber \\
&\tilde{w}_{12,i} = w_{11,i}-2\zeta\omega w_{12,i}-\omega^{2}w_{22,i}, \label{wtilde}\\
&\tilde{w}_{22,i} = 2w_{12,i}-4\zeta\omega w_{22,i}. \nonumber
\end{align}
We rewrite equations (\ref{etassm}) and (\ref{eta_dotssm}) using the Kronecker
product notation
\begin{align}
\eta_{i}&=\left\langle \left[\begin{array}{cccc}
w_{11,i} & w_{12,i} & w_{12,i} & w_{22,i}\end{array}\right]^{\top},\mathbf{z}^{\otimes2}\right\rangle+\varepsilon \bar{W}_i(\phi)
+\mathcal{O}(|(x,\dot{x})|^3,\varepsilon|(x,\dot{x})|,\varepsilon^2),\\
\dot{\eta}_{i}&=\left\langle \left[\begin{array}{cccc}
\tilde{w}_{11,i} & \tilde{w}_{12,i} & \tilde{w}_{12,i} & \tilde{w}_{22,i}\end{array}\right]^{\top},\mathbf{z}^{\otimes2}\right\rangle+\varepsilon \partial_\phi\bar{W}_i(\phi)\Omega +\mathcal{O}(|(x,\dot x)|^3,\varepsilon|(x,\dot{x})|,\varepsilon^2),
\end{align}
where
\begin{equation}
\mathbf{z}=\left[x,\dot{x}\right]^{\top}\in\mathbb{R}^{2},\quad\mathbf{z}^{\otimes2}=\left[x^2,x\dot{x},\dot{x}x,\dot{x}^2\right]^\top\in\mathbb{R}^{4}.
\end{equation}
The modal non-modeling positions and velocities can be written in
vector form as
\begin{align}
&\boldsymbol{\eta}=\underbrace{\left[\begin{array}{cccc}
\mtrx{{w}}_{11} & \mtrx{{w}}_{12} & \mtrx{{w}}_{12} & \mtrx{{w}}_{22}\\
\end{array}\right]}_{\mathbf{W}}\mathbf{z}^{\otimes2}+\varepsilon \bar{\mtrx{W}}(\phi)
+\mathcal{O}(|(x,\dot{x})|^3,\varepsilon|(x,\dot{x})|,\varepsilon^2), \label{wmat}\\
&\boldsymbol{\dot{\eta}}=\underbrace{\left[\begin{array}{cccc}
\mtrx{\tilde{w}}_{11} & \mtrx{\tilde{w}}_{12} & \mtrx{\tilde{w}}_{12} & \mtrx{\tilde{w}}_{22}\\
\end{array}\right]}_{\mathbf{\tilde{W}}}\mathbf{z}^{\otimes2}+\varepsilon \partial_\phi\bar{\mtrx{W}}(\phi)\Omega +\mathcal{O}(|(x,\dot x)|^3,\varepsilon|(x,\dot{x})|,\varepsilon^2).  \label{wtildemat}
\end{align}
The reduced model is obtained by substituting $\mtrx y = \gmtrx{\Phi}_2 \gmtrx{\eta}$ and $\mtrx{\dot{y}}=\gmtrx{\Phi}_2 \gmtrx{\dot{\eta}}$ into equation (\ref{xeq})
  \begin{equation}
    \begin{split}
     & m \ddot x + c \dot x+ k x +P(x,\dot x, \gmtrx{\Phi}_2 \gmtrx{\eta},\gmtrx{\Phi}_2 \gmtrx{\dot{\eta}})+\mathcal{O}(|(x,\dot x)|^4)=\varepsilon F_1 \sin(\phi)
    \end{split} \label{SSMredmodel}
  \end{equation}

Note that the $\mtrx y$ and $\mtrx{\dot{y}}$ dependency in the cubic part of $P$ leads to at least fourth-order terms in ($x, \dot x$), therefore they are omitted from the third-order reduced model. The polynomial function $P$, can explicitly be written as

\begin{align}
\begin{split}
  P = & p_{20\mtrx{0}\mtrx{0}}x^2+p_{11\mtrx{0}\mtrx{0}}x \dot x+p_{02\mtrx{0}\mtrx{0}} \dot x^2 \\
 & + \left<\left(\gmtrx{\Phi}_2\mtrx{W}\right)^\top\left(x\mtrx{p}_{10\mtrx{I}\mtrx{0}}+\dot{x}\mtrx{p}_{01\mtrx{I}\mtrx{0}}\right),\mtrx{z}^{\otimes 2} \right> \\
 & + \left<\left(\gmtrx{\Phi}_2\mtrx{\tilde{W}}\right)^\top\left(x\mtrx{p}_{10\mtrx{0}\mtrx{I}}+\dot{x}\mtrx{p}_{01\mtrx{0}\mtrx{I}}\right),\mtrx{z}^{\otimes 2} \right>\\
&+ p_{30\mtrx{0}\mtrx{0}}x^3+p_{21\mtrx{0}\mtrx{0}}x^2 \dot x+p_{12\mtrx{0}\mtrx{0}}x \dot x^2+p_{03\mtrx{0}\mtrx{0}} \dot x^3+\mathcal{O}(|(x,\dot x)|^4,\varepsilon|(x,\dot{x})|,\varepsilon^2).
\end{split}\label{P2s}
\end{align}
Substituting (\ref{P2s}) into equation (\ref{SSMredmodel}) yields the SSM-reduced model stated in Theorem \ref{thrmssm}.\qed

\section{Proof of Theorem \ref{thm:backbone}} \label{app:backbone}
We rescale equation (\ref{perturbed}) for small vibrations,
    \begin{equation}
      \ddot x + \omega_0^2 x +\epsilon p^{}_{20}x^2 + \epsilon^2\left((p^{}_{30}+\hat p_{30}) x^3+\hat s_{12} x \dot x^2\right) = 0.
      \label{perturbationproblem}
    \end{equation}
    Theorem \ref{thm_ham} guarantees the existence of periodic orbits in the third order reduced dynamics around the fixed point. We extract the backbone curve from the periodic orbits by continuing these orbits from the linear equation as an unperturbed limit of equation (\ref{perturbationproblem}).

We denote the, $\epsilon$ dependent, angular frequency with
      \begin{equation}
        \omega (\epsilon )= \text{$\omega _0$}+\epsilon \text{$\omega_1$} + \epsilon ^2\text{$\omega _2$}+ \mathcal{O}(\epsilon^3),
      \end{equation}
    where $\omega_0$ is the linear angular frequency. Introducing a rescaling in time as $\tau=\omega(\epsilon) t$, we obtain
    \begin{equation}
       \omega^2 (\epsilon )x'' + \omega_0^2 x +\epsilon p^{}_{20}x^2 + \epsilon^2\left((p^{}_{30}+\hat p_{30}) x^3+\hat s_{12} \omega (\epsilon ) x  x'^{2}\right) = 0, \label{eq:scale}
    \end{equation}
    where the prime denotes differentiation with respect to the rescaled time. We expand $x(\tau)$ in $\epsilon$,
    \begin{equation}
      x(\tau)=\phi_0 (\tau)+\epsilon\phi_1 (\tau)+\epsilon^2 \phi_2 (\tau)+\mathcal{O}(\epsilon^3),
    \end{equation}
and substitute the result into equation (\ref{eq:scale}), after which we set $\phi(0)=r, \phi'(0)=0$ to define an initial value problem. Collecting terms of $\mathcal{O}(1)$, satisfying the initial conditions, leads to
    \begin{equation}
      \phi_0(\tau )=r \cos(\tau).
    \end{equation}
Collecting terms of $\mathcal{O}(\epsilon)$, and substituting the solution for $\phi_0(\tau)$, leads to the following equation
    \begin{equation}
      {\omega_0}^2 {\phi_1}''(\tau )+{\omega_0}^2 {\phi_1}(\tau)=\frac{p^{}_{20}r^2}{2} \left( \cos (2 \tau )+1\right)-2 {r} {\omega_0} {\omega_1} \cos (\tau ). \label{eq:ordereps}
    \end{equation}
We set $\omega_1$ in equation (\ref{eq:ordereps}) equal to $0$ in order to avoid having an external resonance, which would lead to an aperiodic solution. This results in the following solution for $\phi_1(\tau)$
    \begin{equation}
      {\phi_1}(\tau)=-\frac{p^{}_{20} r^2}{2 {\omega_0}^2}+\frac{p^{}_{20} r^2 \cos (\tau )}{3 {\omega_0}^2}+\frac{p^{}_{20} r^2 \cos (2 \tau )}{6 {\omega_0}^2}.
    \end{equation}
Finally, collecting terms of $\mathcal{O}(\epsilon^2)$, we obtain the resulting equation for $\phi_2(\tau)$, where we immediately observe that in order to avoid having an external resonance, we must have that
    \begin{equation}
      \omega_2=\frac{9(p^{}_{30}+\hat p_{30})\omega_0^2-10p^{2}_{20}+3\hat s_{12}\omega_0^4}{24\omega_0^3}r^2,
    \end{equation}
which completes the second order approximation of the backbone curve and proves Theorem \ref{thm:backbone}. \qed

\bibliography{biblio.bib}

\begin{thebibliography}{10}
\expandafter\ifx\csname url\endcsname\relax
  \def\url#1{\texttt{#1}}\fi
\expandafter\ifx\csname urlprefix\endcsname\relax\def\urlprefix{URL }\fi
\expandafter\ifx\csname href\endcsname\relax
  \def\href#1#2{#2} \def\path#1{#1}\fi

\bibitem{lsm}
A.~F. Kelley, {Analytic two-dimensional subcenter manifolds for systems with an
  integral}, Pacific J. of Mathematics 29 (1969) 335--350.
\newblock \href {http://dx.doi.org/10.2140/pjm.1969.29.335}
  {\path{doi:10.2140/pjm.1969.29.335}}.

\bibitem{nnmssm}
G.~Haller, S.~Ponsioen, {Nonlinear normal modes and spectral submanifolds:
  existence, uniqueness and use in model reduction}, Nonlinear Dyn 86~(1493).
\newblock \href {http://dx.doi.org/10.1007/s11071-016-2974-z}
  {\path{doi:10.1007/s11071-016-2974-z}}.

\bibitem{shaw1993normal}
S.~Shaw, C.~Pierre, {Normal modes for non-linear vibratory systems}, Journal of
  sound and vibration 164~(1) (1993) 85--124.
\newblock \href {http://dx.doi.org/10.1006/jsvi.1993.1198}
  {\path{doi:10.1006/jsvi.1993.1198}}.

\bibitem{Ponsioen2018}
S.~Ponsioen, T.~Pedergnana, G.~Haller, Automated computation of autonomous
  spectral submanifolds for nonlinear modal analysis, J. Sound Vib. 420 (2018)
  269--295.
\newblock \href {http://dx.doi.org/10.1016/j.jsv.2018.01.048}
  {\path{doi:10.1016/j.jsv.2018.01.048}}.

\bibitem{touze2006nonlinear}
C.~Touz{\'e}, M.~Amabili, {Nonlinear normal modes for damped geometrically
  nonlinear systems: Application to reduced-order modelling of harmonically
  forced structures}, Journal of sound and vibration 298~(4-5) (2006) 958--981.
\newblock \href {http://dx.doi.org/10.1016/j.jsv.2006.06.032}
  {\path{doi:10.1016/j.jsv.2006.06.032}}.

\bibitem{neild2010applying}
S.~Neild, D.~Wagg, {Applying the method of normal forms to second-order
  nonlinear vibration problems}, Proceedings of the Royal Society A:
  Mathematical, Physical and Engineering Sciences 467~(2128) (2010) 1141--1163.
\newblock \href {http://dx.doi.org/10.1098/rspa.2010.0270}
  {\path{doi:10.1098/rspa.2010.0270}}.

\bibitem{neild2015use}
S.~Neild, A.~Champneys, D.~Wagg, T.~Hill, A.~Cammarano, {The use of normal
  forms for analysing nonlinear mechanical vibrations}, Philosophical
  Transactions of the Royal Society A: Mathematical, Physical and Engineering
  Sciences 373~(2051) (2015) 20140404.
\newblock \href {http://dx.doi.org/10.1098/rsta.2014.0404}
  {\path{doi:10.1098/rsta.2014.0404}}.

\bibitem{modalderivatives}
C.~Sombroek, L.~Renson, P.~Tiso, G.~Kerschen, {Bridging the gap between
  nonlinear normal modes and modal derivatives}, in: Nonlinear Dynamics, Volume
  1, Springer, 2016, pp. 349--361.
\newblock \href {http://dx.doi.org/10.1007/978-3-319-15221-9_32}
  {\path{doi:10.1007/978-3-319-15221-9_32}}.

\bibitem{Geradin2014}
M.~G{\'e}radin, D.~Rixen, {Mechanical vibrations: theory and application to
  structural dynamics}, John Wiley \& Sons, 2014.
\newblock \href {http://dx.doi.org/10.1017/aer.2018.27}
  {\path{doi:10.1017/aer.2018.27}}.

\bibitem{thomas}
T.~Breunung, G.~Haller, {Explicit Backbone Curves from Spectral Submanifolds of
  Forced-Damped Nonlinear Mechanical Systems}, arXiv 1709~(05947).
\newblock \href {http://dx.doi.org/10.1098/rspa.2018.0083}
  {\path{doi:10.1098/rspa.2018.0083}}.

\bibitem{equivduffing}
V.~Denis, M.~Jossic, C.~Giraud-Audine, B.~Chomette, A.~Renault, O.~Thomas,
  {Identification of nonlinear modes using phase-locked-loop experimental
  continuation and normal form}, Mech. Syst. Sig. Process.\href
  {http://dx.doi.org/10.1016/j.ymssp.2018.01.014}
  {\path{doi:10.1016/j.ymssp.2018.01.014}}.

\bibitem{slowfast}
G.~Haller, S.~Ponsioen, {Exact Model Reduction by a Slow-Fast Decomposition of
  Nonlinear Mechanical Systems}, Nonlinear Dynamics 90~(06210) (2017) 617--647.
\newblock \href {http://dx.doi.org/10.1007/s11071-017-3685-9}
  {\path{doi:10.1007/s11071-017-3685-9}}.

\bibitem{londono2015identification}
J.~M. Londo{\~n}o, S.~A. Neild, J.~E. Cooper, {Identification of backbone
  curves of nonlinear systems from resonance decay responses}, J. Sound Vib.
  348 (2015) 224--238.
\newblock \href {http://dx.doi.org/10.1016/j.jsv.2015.03.015}
  {\path{doi:10.1016/j.jsv.2015.03.015}}.

\bibitem{touze}
C.~Touz{\'e}, O.~Thomas, A.~Chaigne, Hardening/softening behaviour in
  non-linear oscillations of structural systems using non-linear normal modes,
  Journal of Sound and Vibration 273~(1) (2004) 77 -- 101.
\newblock \href {http://dx.doi.org/10.1016/j.jsv.2003.04.005}
  {\path{doi:10.1016/j.jsv.2003.04.005}}.

\bibitem{Stepanov1960}
V.~Stepanov, V.~Nemytskii, {Qualitative theory of differential equations},
  Princeton Univ. Press, 1960.
\newblock \href {http://dx.doi.org/10.2307/3611670}
  {\path{doi:10.2307/3611670}}.

\bibitem{julia}
J.~Bezanson, S.~Karpinski, V.~B. Shah, A.~Edelman, {Julia: A Fast Dynamic
  Language for Technical Computing}, arXiv 1209~(5145).

\bibitem{diffeqjl}
C.~Rackauckas, Q.~Nie, {DifferentialEquations.jl -- A Performant and
  Feature-Rich Ecosystem for Solving Differential Equations in Julia}, Journal
  of Open Research Software 5~(1).
\newblock \href {http://dx.doi.org/10.5334/jors.151}
  {\path{doi:10.5334/jors.151}}.

\bibitem{toplitz}
S.~Noschese, L.~Pasquini, L.~Reichel, {Tridiagonal Toeplitz matrices:
  Properties and novel applications}, Numerical Linear Algebra with
  Applications 20~(2) (2013) 302.
\newblock \href {http://dx.doi.org/10.1002/nla.1811}
  {\path{doi:10.1002/nla.1811}}.

\bibitem{coco}
H.~Dankowicz, F.~Schilder, Recipes for continuation, SIAM, 2013.
\newblock \href {http://dx.doi.org/10.1137/1.9781611972573}
  {\path{doi:10.1137/1.9781611972573}}.

\bibitem{KenCarp4}
C.~Kennedy, M.~Carpenter, Additive runge–kutta schemes for
  convection–diffusion–reaction equations, Applied Numerical Mathematics
  44~(1) (2003) 139 -- 181.
\newblock \href {http://dx.doi.org/10.1016/s0168-9274(02)00138-1}
  {\path{doi:10.1016/s0168-9274(02)00138-1}}.

\bibitem{Vern7}
J.~H. Verner, {Numerically optimal Runge--Kutta pairs with interpolants},
  Numerical Algorithms 53~(2) (2010) 383--396.
\newblock \href {http://dx.doi.org/10.1007/s11075-009-9290-3}
  {\path{doi:10.1007/s11075-009-9290-3}}.

\end{thebibliography}

\end{document}